\documentclass{amsart}
\usepackage{mathrsfs,yfonts}
\usepackage{amssymb,array}

\newcounter{lemma}
\newtheorem{Theorem}{Theorem}

\newtheorem*{TheoremA}{Theorem A}
\newtheorem*{TheoremAp}{Theorem A'}
\newtheorem*{TheoremB}{Theorem B}
\newtheorem{Lemma}[lemma]{Lemma}
\newtheorem{Corollary}[lemma]{Corollary}
\newtheorem{Proposition}[lemma]{Proposition}

\newtheorem*{Conjecture}{Conjecture}

\def\H{\mathbb H}

\def\Q{\mathbb Q}
\def\Z{\mathbb Z}
\def\R{\mathbb R}
\def\C{\mathbb C}
\def\O{\mathcal O}
\def\CC{\mathscr C}
\def\FF{\mathscr F}
\def\DD{\mathscr D}
\def\GG{\mathscr G}
\def\sg{\mathnormal g}
\def\div{\operatorname{div}}

\def\Mod{\,\mathrm{mod}\,}

\begin{document}

\title{Modular units and cuspidal
  divisor class groups of $X_1(N)$}
\author{Yifan Yang}
\address{Department of Applied Mathematics \\
  National Chiao Tung University \\
  Hsinchu 300 \\
  TAIWAN}
\email{yfyang@math.nctu.edu.tw}
\date\today

\begin{abstract}
  In this article, we consider the group $\FF_1^\infty(N)$ of modular
  units on $X_1(N)$ that have divisors supported on the cusps lying
  over $\infty$ of $X_0(N)$, called the $\infty$-cusps. For each
  positive integer $N$, we will give an explicit basis for the group
  $\FF_1^\infty(N)$. This enables us to compute the group structure of
  the rational torsion subgroup $\CC_1^\infty(N)$ of the Jacobian
  $J_1(N)$ of $X_1(N)$ generated by the differences of the
  $\infty$-cusps. In addition, based on our numerical
  computation, we make a  conjecture on the structure of the
  $p$-primary part of $\CC_1^\infty(p^n)$ for a regular prime $p$.
\end{abstract}

\maketitle

\begin{section}{Introduction} Let $\Gamma$ be a congruence subgroup of
  $SL(2,\Z)$. A \emph{modular unit} $f(\tau)$ on $\Gamma$ is a
  modular function on $\Gamma$ whose poles and zeros are all at
  cusps. If $\{\gamma_j\}$ is a set of right coset representatives of
  $\Gamma$ in $SL(2,\Z)$, then any symmetric sum of $f(\gamma_j\tau)$
  is a modular function on $SL(2,\Z)$ having poles only at $\infty$,
  whence a polynomial of the elliptic $j$-function. Hence, modular units
  are contained in the integral closure of the ring $\C[j]$. For
  arithmetic considerations, one may restrict his attention to the
  modular units that are in the integral closure of $\Q[j]$ or that of
  $\Z[j]$.

  When $\Gamma=\Gamma_0(N)$, M. Newman \cite{Newman1,Newman2}
  determined sufficient conditions for a product
  $\prod_{d|N}\eta(d\tau)^{r_d}$ of Dedekind eta functions to be
  modular on $\Gamma_0(N)$. Because of the infinite product
  representation $\eta(\tau)=q^{1/24}\prod_{n=1}^\infty(1-q^n)$,
  $q=e^{2\pi i\tau}$, such a modular function is a modular unit on
  $\Gamma_0(N)$. In \cite{Takagi}, Takagi showed that for squarefree
  integers $N$, these functions generate the group of modular units on
  $\Gamma_0(N)$. 

  When $\Gamma=\Gamma(N)$ is the principal congruence subgroup of
  level $N$, the group of modular units on $\Gamma(N)$ has been
  studied extensively by Kubert and Lang in a series of papers
  \cite{Kubert-Lang-I,Kubert-Lang-II,Kubert-Lang-III,Kubert-Lang-IV,
  Kubert-Lang-V}. Their main result states that the group of modular
  units on $\Gamma(N)$ is generated by (products of) the Siegel
  functions, except for $2$-cotorsions in the case when $N$ is
  an even composite integer. (See \cite{Kubert}.) Note that here a
  $2$-cotorsion of a subgroup $H$ of a group $G$ means an element $g$
  in $G$ satisfying $g\not\in H$, but $g^2\in H$.

  When $\Gamma=\Gamma_1(N)$, $N\ge 3$, Kubert and Lang \cite[Chapter
  3]{Kubert-Lang-book} consider a special subgroup $\FF$ of modular
  units that have divisors supported at cusps lying over the cusp $0$
  of $X_0(N)$. A set of representatives for such cusps is
  $$
    \{1/a:~1\le a\le N/2,~(a,N)=1\}.
  $$
  Because the homomorphism given by $\div:\,f\mapsto\div f$ is
  injective modulo $\C^\times$, the maximal possible rank for this
  subgroup modulo $\C^\times$ is $\phi(N)/2-1$. In \cite[Chapter
  3]{Kubert-Lang-book}, Kubert and Lang showed that this special
  subgroup does attain the maximal rank allowed. Moreover, they proved
  that this subgroup is generated by a certain class of the Siegel
  functions. Unlike the case $\Gamma(N)$, non-trivial cotorsions do
  not exist in this case.

  The set of modular units is a very important object in number
  theory because of its wide range of applications. For example, the
  ray class field $K(\mathfrak f)$ of conductor $\mathfrak f$ of an
  imaginary quadratic number field $K$ can be obtained by adjoining
  the values, called the \emph{Siegel-Ramachandra-Robert invariants},
  of a suitable chosen modular unit on $\Gamma(N(\mathfrak f))$ at
  certain imaginary quadratic numbers, where $N(\mathfrak f)$ denotes
  the smallest positive rational integer in $\mathfrak f$. (See
  Ramachandra \cite{Ramachandra}.) The same Siegel-Ramachandra-Robert
  invariants also appeared earlier in the Kronecker limit formula for the
  $L$-functions associated with characters of the ray class group,
  originally due to C. Meyer \cite{Meyer}. Furthermore, in
  \cite{Robert}, Robert showed that suitable products of the
  Siegel-Ramachandra-Robert invariants are units, called
  \emph{elliptic units}, in the ray class field $K(\mathfrak f)$, and
  determined the index of the group of elliptic units in the full unit
  group of $K(\mathfrak f)$ for the case $\mathfrak f$ is a power of
  prime ideal.
  The elliptic units featured prominently in Coates and Wiles' proof
  of (the weak form of) the Birch and Swinnerton-Dyer conjecture for the
  cases where elliptic curves have complex multiplication by the ring
  of integers in an imaginary quadratic field of class number one
  \cite{Coates-Wiles}.

  Another area in which modular units plays a fundamental role is the
  arithmetic of the Jacobian variety of a modular curve. Let $\Gamma$
  be a congruence subgroup of level $N$. Consider the cuspidal
  embedding $i_\infty:P\mapsto[(P)-(\infty)]$, sending a point $P$ to
  the divisor class of $(P)-(\infty)$, of the modular curve
  $X(\Gamma)$ into its Jacobian $J(\Gamma)$. From the work of Manin
  and Drinfeld \cite{Manin}, we know that if $P$ is a cusp, then
  $i_\infty(P)$ is a torsion point on $J(\Gamma)$. When the congruence
  subgroup $\Gamma$ is $\Gamma_0(N)$ for some squarefree integer $N$,
  these points are actually rational points in
  $J_0(N)=J(\Gamma_0(N))$. For the case $N=p$ is a prime, Ogg
  \cite{Ogg-survey} conjectured and Mazur \cite{Mazur} proved that the
  full rational torsion subgroup of $J_0(p)$ is generated by
  $i_\infty(0)=(0)-(\infty)$. Later on, in light of the Manin-Mumford
  conjecture (theorem of Raynaud), Coleman, Kaskel, and Ribet
  \cite{Coleman} proposed a stronger conjecture asserting that
  with the exception of hyperelliptic cases (excluding $p=37$), if a
  modular curve $X_0(p)$ has genus greater than $1$, then the only
  points on $X_0(p)$ that are mapped into torsion subgroups of
  $J_0(p)$ under $i_\infty$ are the two cusps. This conjecture was
  later established by Baker \cite{Baker}.

  For congruence subgroups of type $\Gamma_1(N)$, the images of cusps
  under the cuspidal embedding $i_\infty$ are rational over
  $\Q(e^{2\pi i/N})$. Moreover, when the cusps $P$ are lying over
  $\infty$ of $X_0(N)$, called the $\infty$-cusps, the
  images $i_\infty(P)$ are rational over $\Q$. In particular, these
  points $i_\infty(P)$ generate a rational torsion subgroup of
  $J_1(N)$. We shall denote this rational torsion subgroup by
  $\CC_1^\infty(N)$. If we let $C_1^\infty(N)$ denote the set of
  $\infty$-cusps, $\DD_1^\infty(N)$ the group of divisors of degree
  $0$ generated by $C_1^\infty(N)$, and $\FF_1^\infty(N)$ the group of
  modular units with divisors supported on $C_1^\infty(N)$, then the
  order of the $\CC_1^\infty(N)$ is simply the divisor class
  number $h_1^\infty(N):=|\DD_1^\infty(N)/\div\FF_1^\infty(N)|$, which
  has been completely determined in literature. Firstly, the prime
  cases were settled by Klimek \cite{Klimek}. Then Kubert and Lang
  \cite{Kubert-Lang-classno} considered the prime power
  cases $p^n$ with $p\ge 5$. Finally, Yu \cite{Yu} gave a class number
  formula in full generality (given as Theorem A in the next section).
  Note that the divisor class number formulas in
  \cite[Theorem 3.4]{Kubert-Lang-book} and \cite{Yu} are stated as for
  the subgroup generated by the differences of the cusps lying over
  $0$ of $X_0(N)$. However, it is plain that the Atkin-Lehner
  involution
  $\left(\begin{smallmatrix}0&-1\\N&0\end{smallmatrix}\right)$ gives
  rise to an isomorphism between the two divisor class groups.

  It worths mentioning that the class number formula
  $$
    h_1^\infty(p^n)=\prod_{\chi\neq\chi_0}\frac14B_{2,\chi}\cdot
    \begin{cases}p^{p^{n-1}-2n+2}, &p\text{ odd}, \\
    2^{2^{n-1}-2n+3}, &p=2,
    \end{cases}
  $$
  for the prime power cases suggests that there exists an Iwasawa
  theory of $\Z_p$-extensions for the $p$-parts of
  $\CC_1^\infty(p^n)$. The $\Z_p$-extension cannot be coming from the
  covering $X_1(p^{n+1})\to X_1(p^n)$ as the covering is not even Galois.
  However, it is still possible to define a $\Z_p$-action on the sets
  $C_1^\infty(p^n)$ of $\infty$-cusps. This in turn induces a
  $\Z_p$-action on the divisor class groups $\CC_1^\infty(p^n)$. We
  plan to address the Iwasawa properties of the $p$-parts of
  $\CC_1^\infty(p^n)$ in the future, cf. the conjecture below.

  With the order of $\CC_1^\infty(N)$ determined, it is natural to ask
  the following questions:
  \begin{enumerate}
  \item Is $\CC_1^\infty(N)$ the full rational torsion subgroup of
  $J_1(N)$ when the $\infty$-cusps are the only cusps whose images in
  $J_1(N)$ are rational over $\Q$?
  \item What is the group structure of $\CC_1^\infty(N)$? In particular,
  what is the structure of the $p$-primary part in the case $N=p^n$ is
  a prime power?
  \end{enumerate}
  For the case $N=p$ is a prime, Conrad, Edixhoven, and Stein
  \cite[Section 6.1]{Conrad-Edixhoven-Stein} devised an algorithm to
  obtain an upperbound for the order of the rational torsion subgroup
  of $J_1(p)$. Comparing the upperbound with Kubert and Lang's divisor
  class number formula, they found that for $p<157$ not equal to $29$,
  $97$, $101$, $109$, $113$, the upperbound agrees with the divisor
  class number. In other words, in those cases $\CC_1^\infty(p)$ is
  the whole rational torsion subgroup. Based on this numerical
  evidence, Conrad, Edixhoven, and Stein
  \cite[Page 392]{Conrad-Edixhoven-Stein} conjectured that for a prime
  $p\ge 5$, any rational torsion point on $J_1(p)$ is contained in the
  subgroup generated by the images of $\infty$-cusps under $i_\infty$.

  For the second question about the structure of $\CC_1^\infty(N)$,
  the problem eventually boils down to the problem of finding a basis
  for the group $\FF_1^\infty(N)$ of modular units. In \cite{Csirik}
  Csirik determined the structure of $\CC_1^\infty(p)$ for the first few
  primes $p$. Presumably, he should have a method or an algorithm for
  finding a basis for $\FF_1^\infty(p)$. Beyond \cite{Csirik}, it
  appears that nothing is known in literature.

  In this article, for each integer $N$ greater than $4$, we will
  construct an explicit basis for the group $\FF_1^\infty(N)$ in terms
  of the Siegel functions. We will use two different methods in
  obtaining our results, depending on whether $N$ is a prime power.
  When $N=p^n$ is a prime power, we use Yu's class number formula
  (Theorem A below) and linear algebra arguments to show that our
  basis does generate the whole group of modular units. When $N$ is
  not a prime power, we use Yu's characterization of $\FF_1^\infty(N)$
  (Theorem B below) to argue directly that every function in
  $\FF_1^\infty(N)$ is a product of functions from our basis. In
  either case, the so-called distribution relations (Lemma \ref{lemma:
  Bernoulli relation}) play an essential role.

  We then carry out some numerical computation on the structure of
  $\CC_1^\infty(N)$ for $N$ in the range of a few hundreds. This amounts
  to computing the Smith normal form of the matrix representing the
  divisors of the modular units in the basis. We have also computed
  the structure of the $p$-primary part of $\CC_1^\infty(p^n)$ for
  $p^n<800$ with $p=2,3,5,7$. Based on our limited numerical data (see
  Section \ref{section: results}), we make the following
  conjecture about the structure of the $p$-primary part of
  $\CC_1^\infty(p^n)$ for a regular prime $p$.

\begin{Conjecture} Let $p$ be a prime and $n$ be a positive
  integer. Consider the endomorphism $[p]:\CC_1^\infty(p^n)\to
  \CC_1^\infty(p^n)$ defined by multiplication by $p$. Define the
  \emph{$p$-rank} of $\CC_1^\infty(p^n)$ to be the integer $k$ such
  that the kernel of $[p]$ has $p^k$ elements. If $p$ is a regular
  prime, then we conjecture that the $p$-rank of $\CC_1^\infty(p^n)$ is
  $$
    \frac12(p-1)p^{n-2}-1
  $$
  for prime power $p^n\ge 8$ with $n\ge 2$.

  More precisely, for a prime power $p^n\ge 8$ with a regular prime
  $p$, we conjecture that the the number of copies of $\Z/p^{2k}\Z$ in
  the primary decomposition of $\CC_1^\infty(p^n)$ is
  $$
    \begin{cases}
    \frac12(p-1)^2p^{n-k-2}-1, &\text{if }p=2\text{ and }k\le n-3, \\
    \frac12(p-1)^2p^{n-k-2}-1, &\text{if }p\ge3\text{ and }k\le n-2, \\
    \frac12(p-5), &\text{if }p\ge 5\text{ and }k=n-1, \\
    0, &\text{else}.
    \end{cases}
  $$
  and the number of copies of $\Z/p^{2k-1}\Z$ is
  $$
    \begin{cases}
    1, &\text{if }p=2\text{ and }k\le n-3, \\
    1, &\text{if }p=3\text{ and }k\le n-2, \\
    1, &\text{if }p\ge 5\text{ and }k\le n-1, \\
    0, &\text{else}.
    \end{cases}
  $$
\end{Conjecture}

  For powers of irregular primes, since the first case $N=37^2$ has
  already exceeded our computational capacity, we do not have any
  numerical data to make any informed conjecture. However, we expect
  that the structure of the $p$-primary part of $\CC_1^\infty(p^n)$ for
  such cases should still possess similar patterns. It would be
  very interesting to see if the same phenomenon also appears in the
  case of $\Z_p$-extension of number fields and function fields over
  finite fields.

  The rest of the article is organized as follows. In Section
  \ref{section: methodology}, we explain our methods in details. In
  particular, we will review the basic properties of the Siegel
  functions, and then describe how we put Yu's theorems to
  use. In Sections \ref{section: prime powers} and
  \ref{section: non-prime power}, we consider the
  prime power cases and non-prime power cases, respectively. The prime
  cases, the odd prime power cases, and the power of $2$ cases are
  dealt with separately in Theorems \ref{theorem: prime},
  \ref{theorem: prime power}, and \ref{theorem: power of 2}. In
  Theorem \ref{theorem: squarefree} of Section
  \ref{section: non-prime power}, we consider the squarefree 
  composite cases. Finally, in Theorem \ref{theorem: nonsquarefree} we
  present bases for the remaining cases. Note that even though
  Theorem \ref{theorem: squarefree} is just a special case of Theorem
  \ref{theorem: nonsquarefree}, because the construction of bases in
  those cases is signicantly simpler and more transparent, those cases
  are handled separately to give the reader a clearer picture. Because
  of the complexity of our bases, many examples will be
  given. Finally, in Section 5, we give the results of our computation
  on the structures of of $\CC_1^\infty(N)$ for $N\le 100$. The
  structures of the $p$-primary parts of $\CC_1^\infty(p^n)$ are also
  given for $p^n\le 800$ with $p=2,3,5,7$, along with a few cases of
  $\CC_1^\infty(mp^n)$.
\medskip

\centerline{\sc Acknowledgments}
\medskip

The author would like to thank Professor Jeng-Daw Yu of the Queen's
University for his interest in the work and explaining Iwasawa's
theory of $\Z_p$-extension to the author. The author would also like
to thank Professor Jing Yu of the National Tsing Hua University,
Taiwan and Professor Imin Chen of the Simon Fraser University for
clarifying questions about arithmetic properties of modular curves.

Part of this work was done while the author was a visiting scholar at
the Queen's University, Canada. The author would like to thank
Professor Yui and the Queen's University for their warm hospitality.
The visit was supported by Grant 96-2918-I-009-005 of the National
Science Council of Taiwan and Discovery Grant of Professor Yui of the
National Sciences and Engineering Council of Canada. The author was
also partially supported by Grant 96-2628-M-009-014 of the National
Science Council.
\end{section}

\begin{section}{Methodology} \label{section: methodology}
\begin{subsection}{Notations and conventions}
  \label{subsection: notations} In this section, we will
  collect all the notations and conventions that will be used
  throughout the paper.

  Foremost, the letter $N$ will be reserved exclusively for the
  level of the congruence subgroup $\Gamma_1(N)$ that is under our
  consideration. For $N\ge 5$, we let
  \begin{align*}
    C_1^\infty(N)&=\text{the set of cusps of }X_1(N)
      \text{ lying over }\infty\text{ of }X_0(N), \\
    \DD_1^\infty(N)&=\text{the group of divisors of degree }0
      \text{ on }X_1(N)\text{ having support on }C_1^\infty(N), \\
    \FF_1^\infty(N)&=\text{the group of modular units on }X_1(N) \\
    &\quad\qquad \text{ having divisors supported on }C_1^\infty(N),
      \\
    \CC_1^\infty(N)&=\text{the divisor class group }
      \DD_1^\infty(N)/\div\FF_1^\infty(N), \\
    h_1^\infty(N)&=|\DD_1^\infty(N)/\div\FF_1^\infty(N)|,
      \text{ the divisor class number}.
  \end{align*}
  Given a Dirichlet character $\chi$ modulo $N$, we let $B_{k,\chi}$
  denote the generalized Bernoulli numbers defined by the power series
  $$
    \sum_{a=1}^N\chi(a)\frac{te^{at}}{e^{Nt}-1}
   =\sum_{k=0}^\infty\frac{B_{k,\chi}}{k!}t^k.
  $$
  In particular, if we let $\{x\}$ be the fractional part of a real
  number $x$, then we have
  $$
    B_{2,\chi}=N\sum_{a=1}^N\chi(a)B_2(a/N),
  $$
  where
  $$
    B_2(x)=\{x\}^2-\{x\}+\frac16
  $$
  is the second Bernoulli function. Note that some authors define
  generalized Bernoulli numbers slightly differently. If an
  imprimitive Dirichlet character $\chi$ has conductor $f$ and let
  $\chi_f$ be the character modulo $f$ that induces $\chi$, then in
  \cite{Yu} the Bernoulli number associated with $\chi$ is defined as
  \begin{equation} \label{equation: Yu's Bernoulli}
    \frac f2\sum_{a=1}^f \chi_f(a)B_2(a/f),
  \end{equation}
  which is $B_{2,\chi_f}/2$ in our notation. The numbers $B_{2,\chi}$
  and $B_{2,\chi_f}$ are connected by the formula
  $$
    B_{2,\chi}=B_{2,\chi_f}\prod_{p|N,~p\nmid f}(1-\chi_f(p)p),
  $$
  proved in \eqref{equation: conductor relation} below.

  We now introduce the Siegel functions. For $a=(a_1,a_2)\in\Q^2$,
  $a\not\in\Z^2$, the Siegel functions $\sg_a(\tau)$ are usually
  defined in terms of the Klein forms and the Dedekind eta functions.
  For our purpose, we only need to know that they have the following
  infinite product representation. Setting $z=a_1\tau+a_2$,
  $q_\tau=e^{2\pi i\tau}$, and $q_z=e^{2\pi iz}$, we have
  $$
    \sg_a(\tau)=-e^{2\pi ia_2(a_1-1)/2}q_\tau^{B(a_1)/2}
    (1-q_z)\prod_{n=1}^\infty(1-q_\tau^n q_z)(1-q_\tau^n/q_z),
  $$
  where $B(x)=x^2-x+1/6$ is the second Bernoulli polynomial.
  These functions, without the exponential factor $-e^{2\pi
  ia_2(a_1-1)/2}$, are sometimes referred to as generalized Dedekind
  eta functions by some authors, notably
  \cite[Chapter VIII]{Schoeneberg}. 

  For a given integer $N$, we also define a class of functions
  $$
    E_a^{(N)}(\tau)=-\sg_{(a/N,0)}(N\tau)
   =q^{NB(a/N)/2}\prod_{n=1}^\infty\left(1-q^{(n-1)N+a}\right)
    \left(1-q^{nN-a}\right),
  $$
  for integers $a$ not congruent to $0$ modulo $N$, where
  $q=e^{2\pi i\tau}$. The basic properties of $E_a^{(N)}(\tau)$ will
  be introduced in the next section. They will be the building blocks
  of our bases. In a loose sense, we say these functions are functions of
  \emph{level $N$}. (It does not mean that $E_a^{(N)}$ by itself is
  modular on some congruence subgroup of level $N$, but that we use
  them to construct modular functions on $\Gamma_1(N)$. However, the
  $12N$th power of $E_a^{(N)}$ is indeed modular on $\Gamma_1(N)$.)
  Notice that if $d|(a,N)$, then we have the trivial relation
  $$
    E_a^{(N)}(\tau)=E_{a/d}^{(N/d)}(d\tau).
  $$
  From time to time, we will write $E_{a/d}^{(N/d)}(d\tau)$ in place
  of $E_a^{(N)}(\tau)$ to stress the point that such a function is
  really coming from a congruence subgroup of lower level. Also, the
  notation $E_a(\tau)$ without the superscript has the same meaning
  as $E_a^{(N)}(\tau)$.

  Note that it is easy to check that $E_{g+N}=E_{-g}=-E_g$. Therefore,
  for a given $N$, there are only $\lceil(N-1)/2\rceil$ essentially
  distinct $E_g$, indexed over the set $(\Z/N\Z)/\pm 1-\{0\}$. Thus,
  a product $\prod_g E_g^{e_g}$ is understood to be taken over
  $g\in(\Z/N\Z)/\pm 1-\{0\}$.

  Finally, for a divisor $K$ of $N$, the group $\Z/K\Z$ acts on
  $(\Z/N\Z)/\pm 1$ naturally by $a\mapsto a+kN/K$ for
  $a\in(\Z/N\Z)/\pm1$ and $k\in\Z/K\Z$. We let
  $$
    \O_{a,K}:=\{b\in\Z/N\Z:~a\equiv b\mod N/K\}/\pm 1
  $$
  be the orbit of $a$ under this group action.
\end{subsection}

\begin{subsection}{Properties of Siegel functions}
  In this section, we collect properties of the functions
  $E_g^{(N)}=E_g$ relevant to our consideration. The proof of these
  properties can be found in \cite{Yang}. Throughout the section, $N$
  is a fixed integer greater than $4$.

The first property given is the transformation law for $E_g$.

\begin{Proposition}[{\cite[Corollary 2]{Yang}}]
  \label{proposition: Eg formulas} The functions $E_g$ satisfy
\begin{equation}
\label{shifting for Eg}
  E_{g+N}=E_{-g}=-E_g.
\end{equation}
Moreover, let $\gamma=\begin{pmatrix}a&b\\ cN&d\end{pmatrix}\in
\Gamma_0(N)$. We have, for $c=0$,
$$
  E_g(\tau+b)=e^{\pi ibNB(g/N)}E_g(\tau),
$$
and, for $c>0$,
\begin{equation}\label{Eg formula}
  E_g(\gamma\tau)=\epsilon(a,bN,c,d)e^{\pi i(g^2ab/N-gb)}
  E_{ag}(\tau),
\end{equation}
where
\begin{equation*}
  \epsilon(a,b,c,d)
  =\begin{cases}
    e^{\pi i\left(bd(1-c^2)+c(a+d-3)\right)/6},
      &\text{if }c\text{ is odd}, \\
    -ie^{\pi i\left(ac(1-d^2)+d(b-c+3)\right)/6},
      &\text{if }d\text{ is odd}.
  \end{cases}
\end{equation*}
\end{Proposition}

\noindent{\bf Remark.} The proof of this property given in \cite{Yang}
  used the transformation formula for the Jacobi theta function
  $\vartheta_1(z|\tau)$. An alternative approach will be to use the
  fact that the Siegel function is equal to the product of the Klein
  form and $\eta(\tau)^2$. Then from the fact that the Klein form is a
  form of weight $-1$ on $SL(2,\Z)$ and the transformation formula for
  $\eta(\tau)$ (\cite[pp.~125--127]{Weber}), one can deduce the above
  formula for $E_g$. The advantage of the argument given in
  \cite{Yang} is that it can be generalized to the case of Jacobi
  forms (see \cite{Eichler-Zagier-Jacobi} for the definition and
  properties of Jacobi forms) and the Riemann theta functions of
  higher genus.
\medskip

From the transformation law, we immediately obtain sufficient
conditions for a product $\prod_g E_g^{e_g}$ to be modular on
$\Gamma_1(N)$.

\begin{Proposition}[{\cite[Corollary 3]{Yang}}]
  \label{proposition: Gamma1 invariant} Consider a
  function $f(\tau)=\prod_g E_g(\tau)^{e_g}$, where $g$ and $e_g$ are
  integers with $g$ not divisible by $N$. Suppose that one has
\begin{equation}
\label{Corollary 3: condition 1}
  \sum_g e_g\equiv 0\mod 12, \qquad \sum_g ge_g\equiv 0\mod 2
\end{equation}
and
\begin{equation}
\label{Corollary 3: condition 2}
  \sum_g g^2e_g\equiv 0\mod 2N.
\end{equation}
Then $f$ is a modular function on $\Gamma_1(N)$.

Furthermore, for the cases where $N$ is a positive odd integer,
conditions (\ref{Corollary 3: condition 1}) and
(\ref{Corollary 3: condition 2}) can be reduced to
$$
  \sum_g e_g\equiv 0\mod 12
$$
and
$$
  \sum_g g^2e_g\equiv 0\mod N,
$$
respectively.
\end{Proposition}

The following proposition gives the order of $E_g$ at cusps of $X_1(N)$.

\begin{Proposition}[{\cite[Lemma 2]{Yang}}]
\label{proposition: behavior of Eg}
The order of the function $E_g$ at a cusp $a/c$ of $X_1(N)$ with
$(a,c)=1$ is $(c,N)B_2(ag/(c,N))/2$, where $B_2(x)=\{x\}^2-\{x\}+1/6$
and $\{x\}$ denotes the fractional part of a real number $x$.
\end{Proposition}

From the above proposition, we can easily see that when $N=p$ is a
prime, any quotient $E_{g_1}/E_{g_2}$ will have a divisor (with
rational coefficients) supported on $C_1^\infty(p)$. For general $N$,
after a moment of thought, we find the following condition for a
product $\prod_g E_g^{e_g}$ to have a divisor supported on
$C_1^\infty(N)$.

\begin{Lemma}[{\cite[Lemma 4.1]{Yu}}] \label{lemma: orbit condition}
  Let $N$ be a positive composite integer. If $f(\tau)=\prod_g
  E_g^{e_g}$ satisfies the orbit condition
$$
  \sum_{g\in\O_{a,p}}e_g=0
$$
for all $a\in\Z/N\Z$ and all prime divisors $p$ of $N$, then the poles
and zeros of $f(\tau)$ occur only at cusps lying over $\infty$ of
$X_0(N)$.
\end{Lemma}

In fact, Yu \cite{Yu}, cited as Theorem B below, showed that when $N$
has at least two distinct prime factors, the orbit condition is the
necessary and sufficient condition for $\prod_g E_g^{e_g}$ to be a
modular unit contained in $\FF_1^\infty(N)$. (For prime power cases
$N=p^n$, a function satisfying the orbit condition alone will still
have a divisor supported on $C_1^\infty(p^n)$, but the orders at cusps
may not be integers.)

Finally, we need the following distribution relations among $E_g$ and
among generalized Bernoulli numbers.

\begin{Lemma} \label{lemma: Bernoulli relation} Assume that $N=nM$ for
  some positive integer $n$. Then we have for all integers $a$ with
  $0<a<M$,
\begin{equation} \label{equation: Bernoulli relation}
  N\sum_{k=0}^{n-1}B_2\left(\frac{kM+a}N\right)
 =MB_2\left(\frac aM\right),
\end{equation}
and consequently
$$
  \prod_{k=0}^{n-1}E_{kM+a}^{(N)}(\tau)=E_a^{(M)}(\tau).
$$ 
\end{Lemma}

\begin{proof} The proof of the first statement is a straightforward
  computation. Then the second statement follows immediately from
  \eqref{equation: Bernoulli relation} and the definition of
  $E_a^{(M)}$.
\end{proof}

\noindent{\bf Remark.} Relation \eqref{equation: Bernoulli relation}
 shows that the set of all numbers $NB_2(a/N)$ forms a distribution,
 called the \emph{Bernoulli distribution}. See \cite[Chapter
 12]{Washington} and \cite[Chapter 1]{Kubert-Lang-book}.
\end{subsection}

\begin{subsection}{Method for prime power cases} Let $N=p^n$ be a
  prime power. In this section we will explain how to use the
  following result of Yu to obtain a basis for $\FF_1^\infty(N)$.
  Note that in the formula, if an imprimitive Dirichlet character $\chi$
  modulo $N$ has conductor $f$, we let $\chi_f$ denote the character
  modulo $f$ that induces $\chi$.

\begin{TheoremA}[{\cite[Theorem 5]{Yu}}] \label{theorem: Yu} Let $N\ge
  5$ be a positive integer and $\omega(N)$ be the number of distinct
  prime divisors of $N$. For a prime divisor $p$ of $N$, let
  $p^{n(p)}$ be the highest power of $p$ dividing $N$. We have the
  class number formula
  \begin{equation} \label{equation: Yu's formula}
    h_1^\infty(N)=\prod_{p|N}p^{L(p)}
    \prod_{\chi\neq\chi_0\text{ even}}\left(\frac14B_{2,\chi_f}
    \prod_{p|N}(1-p^2\chi_f(p))\right)
  \end{equation}
  for $h_1^\infty(N)=|\DD_1^\infty(N)/\div\FF_1^\infty(N)|$, where
  $$
    L(p)=\begin{cases}
    \phi(N/p^{n(p)})(p^{n(p)-1}-1)-2n(p)+2, &\text{if }\omega(N)\ge 2, \\
    p^{n-1}-2n+2, &\text{if }N=p^n\text{ and }p\text{ is odd}, \\
    2^{n-1}-2n+3, &\text{if }N=2^n\ge 8,
    \end{cases}
  $$
  and the product $\prod_{\chi}$ runs over all even non-principal
  Dirichlet characters modulo $N$.
\end{TheoremA}

\noindent{\bf Remark.} Note that the definition of the generalized
  Bernoulli numbers used in \cite{Yu} is different from ours. See
  Section \ref{subsection: notations} for details.

In \cite{Yu}, the number $L(p)$ in the class number formula is given
as
$$
  L(p)=\begin{cases}
    \phi(N/p^{n(p)})(p^{n(p)-1}-1)-2n(p)-2,
      &\text{if }N\text{ is composite}, \\
    p^{n-1}-2n-2, &\text{if }N=p^n>4. \\
    \end{cases}
$$
As pointed out in \cite{Hazama}, the minus sign in front of $2$ is an
obvious misprint. Also, the use of the term ``composite'' in \cite{Yu}
is somehow unconventional as it refers to an integer $N$ with
$\omega(N)>1$ throughout \cite{Yu}.

The discrepancy in the case $N=2^n$ is due to a slight oversight in
the proof of Lemma 3.4 of \cite{Yu}. Upon a close examination, one can
see that when $N=2^n$, Lemma 3.4 of \cite{Yu} is valid only in the
range $3\le\ell\le n$, not $2\le\ell\le n$. Note that when $N=2^3$,
Yu's formula actually gives
$$
  2^{2^2-6+2}\cdot\frac14\cdot 8\cdot\left(B_2(1/8)-B_2(3/8)
 -B_2(5/8)+B_2(7/8)\right)=\frac12
$$
as the class number, which clearly
can not be the correct number.
\medskip

\noindent{\bf Remark.} In \cite{Hazama}, Hazama gave another class
number formula for $h_1^\infty(N)$ for $N\not\equiv 0\mod 4$ in terms
of the so-called Demjanenko matrix. The new formula is more or less a
result of manipulation of the generalized Bernoulli numbers. It does
not give a new proof of Yu's formula. Nonetheless, Hazama's formula
will be useful in verifying the correctness of our numerical
computation. However, the reader should be mindful of several errors
when applying Hazama's formula.

In the statement of Theorem 3.1 of \cite{Hazama}, $f_i$ should be
defined as the multiplicative order of $p_i$ in
$(\Z/m_i\Z)^\times/\pm 1$, not $(\Z/m_i\Z)^\times$, and $e_i$ should
be $\phi(m_i)/(2f_i)$. Then $A_{p_i}(m)$ is simply
$(1+p_i^{f_i})^{e_i}/(1+p_i)$. The quantities $f$, $e$, $A(m)$
should be defined analogously. Moreover, as already pointed out in
\cite{Hazama2}, there is a discrepancy in the definition of the generalized
Bernoulli numbers. However, the author of \cite{Hazama2} still missed
the $1/2$ factor in Yu's definition \eqref{equation: Yu's Bernoulli}
of generalized Bernoulli numbers.
\medskip

Yu's formula can be slightly simplified.

\begin{TheoremAp} \label{theorem: Yu formula} Let all the
  notations be given as in Theorem A. For a prime divisor $p$ of $N$,
  we let $f_p$ denote the multiplicative order of $p$ in
  $(\Z/(N/p^{n(p)})\Z)^\times/\pm1$, and
  $e_p=|(\Z/(N/p^{n(p)})\Z)^\times/\pm1|/f_p$. Then the class number
  $h_1^\infty(N)$ is equal to
  \begin{equation} \label{equation: Yu formula 2}
    h_1^\infty(N)=\prod_{p|N}\frac{p^{L(p)}(1+p^{f_p})^{e_p}}{(1+p)}
    \cdot\prod_{\chi\neq\chi_0\text{ even}}\frac14B_{2,\chi},
  \end{equation}
  where the last product is taken over all even non-principal
  Dirichlet characters modulo $N$.
\end{TheoremAp}

\begin{proof} We first determine the relation between $B_{2,\chi}$ and
  $B_{2,\chi_f}$ for a character of conductor $f$. We have, by the
  exclusion-inclusion principle,
  $$
    B_{2,\chi}=N\sum_{a=1}^NB_2(a/N)\chi(a)
    =N\sum_{k|(N/f),~(k,f)=1}\mu(k)\sum_{a=1}^{N/k}B_2(ak/N)\chi_f(ak).
  $$
  The inner sum is equal to
  \begin{equation*}
  \begin{split}
  &\chi_f(k)\sum_{a=0}^{f-1}\chi_f(a)\sum_{m=0}^{N/(kf)-1}
     \left(\frac{(mf+a)^2k^2}{N^2}-\frac{(mf+a)k}N+\frac16\right) \\
  &\qquad=\chi_f(k)\sum_{a=0}^{f-1}\frac{kf}NB_2(a/f)\chi_f(a)
   =\frac kN\chi_f(k)B_{2,\chi_f}.
  \end{split}
  \end{equation*}
  It follows that
  \begin{equation} \label{equation: conductor relation}
    B_{2,\chi}=B_{2,\chi_f}\sum_{k|(N/f),~(k,f)=1}\mu(k)\chi_f(k)k
    =B_{2,\chi_f}\prod_{p|N,~p\nmid f}(1-\chi_f(p)p),
  \end{equation}
  and consequently,
  $$
    h_1^\infty(N)=\prod_{p|N}p^{L(p)}\prod_{\chi\neq\chi_0}
    \frac14B_{2,\chi}\prod_p\frac{1-\chi_f(p)p^2}{1-\chi_f(p)p}.
  $$
  Now for each prime factor $p$ of $N$, there are precisely
  $|(\Z/(N/p^{n(p)})\Z)^\times/\pm 1|$ even Dirichlet characters
  modulo $N$ whose conductors are relatively prime to $p$. As $\chi$
  runs over such characters, the values of $\chi(p)$ go through the
  complete set of $f_p$th roots of unity $e_p$ times. Therefore, for
  $a=1,~2$,
  $$
    \prod_{\chi\neq\chi_0}(1-\chi_f(p)p^a)
   =\frac{(1-p^{af_p})^{e_p}}{1-p^a},
  $$
  and \eqref{equation: Yu formula 2} follows.
\end{proof}

We now describe our method for prime power cases. Let $N=p^n$ be a
prime power greater than $4$, and set $n=\phi(N)/2$. The divisor group
$\DD_1^\infty(N)$ can be naturally embedded in the hyperplane
$x_1+\cdots+x_n=0$ inside $\R^n$ by sending a divisor $\sum c_n(P_n)$
of degree $0$ to $(c_1,\ldots,c_n)$, where $P_i$ denote the cusps in
$C_1^\infty(N)$. It is obvious that the image of $\DD_1^\infty(N)$ is
the lattice $\Lambda$ generated by
$V=\{(0,\ldots,0,1,-1,0,\ldots,0)\}$. Now if $f_1,\ldots,f_{n-1}$ are
multiplicatively independent modular units contained in
$\FF_1^\infty(N)$, then the images of the divisors of $f_i$ will
generate a sublattice $\Lambda'$ of $\Lambda$ of the same rank whose
index in $\Lambda$ can be determined using the following lemma.

\begin{Lemma} \label{lemma: lattice index} Let $\Lambda\subset\R^n$ be
  the lattice of dimension $n-1$ generated by the vectors of the form
  $(0,\ldots,1,-1,0,\ldots,0)$. Let $\Lambda'$ be a sublattice of
  $\Lambda$ of the same rank generated by $v_1,\ldots,v_{n-1}\in\Lambda$.
  Let $v_n=(c_1,\ldots,c_n)$ be any vector such that $\sum_ic_i\neq
  0$, and $M$ be the $n\times n$ matrix whose $i$th row is $v_i$. Then we
  have
  $$
    (\Lambda:\Lambda')=\left|\left(\sum_{i=1}^nc_i\right)^{-1}\det M\right|.
  $$
\end{Lemma}

\begin{proof} In general, to determine the index of a sublattice
  $\Lambda'$ in a lattice $\Lambda$ of codimension $1$ in $\R^n$, we
  pick a nonzero vector $v$ in $\R^n$ that is orthogonal to $\Lambda$,
  and form two matrices $A$ and $A'$, where the rows of $A$ are
  genenerators of $\Lambda$ and $v$, while those of $A'$ are
  generators of $\Lambda'$ and $v$. Then the index of $\Lambda'$ in
  $\Lambda$ is equal to $|\det A'/\det A|$.

  Now for the lattice $\Lambda$ generated by
  $(0,\ldots,1,-1,0,\ldots,0)$, we can choose the vector $v$ to be
  $(1/n,\ldots,1/n)$. Then the determinant of $A$ is $1$. On the other
  hand, for the matrix $M$ in the lemma, by adding suitable multiples
  of the first $n-1$ rows to the last row, we can bring the last row
  into $(0,\ldots,0,c)$. Since the sum of entries in each of the first
  $n-1$ rows is $0$, the number $c$ is equal to the sum $\sum_i
  c_i$. Note that this procedure does not change the determinant. By
  the same token, we can also transform the matrix $A'$ corresponding
  to $\Lambda'$ into a matrix whose first $n-1$ rows are $v_i$ and
  whose last row is $(0,\ldots,0,1)$ without change the determinant.
  From this, we see that $\det M=c\det A'$, and therefore
  $(\Lambda:\Lambda')=|c^{-1}\det M|$.
\end{proof}

Now let us consider the prime case $N=p$ for the moment. Assume that
$f_i(\tau)=\prod_j E_j^{e_{i,j}}$, $i=1,\ldots,(p-1)/2-1$, are
multiplicatively independent modular units in $\FF_1^\infty(p)$. Let
$M$ be the square matrix of size $(p-1)/2$ whose $(j,k)$-entry is the
order of $E_j$ at $k/p$, and $U=(u_{ij})$ be the square matrix of the
same size with
\begin{equation} \label{equation: temp u}
  u_{ij}=\begin{cases}e_{i,j}, &\text{if }1\le i\le(p-3)/2, \\
  0, &\text{if }i=(p-1)/2\text{ and }1\le j\le(p-3)/2, \\
  1, &\text{if }i=j=(p-1)/2.
  \end{cases}
\end{equation}
Then the product $UM$ will have the orders of $f_i$ at $k/p$ as its
$(i,k)$-entry for the first $(p-3)/2$ rows and the last row consists
of the orders of $E_{(p-1)/2}$ at $k/p$, which, by Proposition
\ref{proposition: behavior of Eg} are $pB_2(k(p-1)/p)/2$. Thus, by
Lemma \ref{lemma: lattice index}, the subgroup generated by the
divisors of $f_i$ will have index
$$
  \left|\left(\frac p2\sum_{k=1}^{(p-1)/2}
  B_2\left(\frac{k(p-1)}p\right)\right)^{-1}\det U\det M\right|
$$
in the full divisor group. In particular, $\{f_i\}$ generates
$\FF_1^\infty(p)$ if and only if this number is equal to
$h_1^\infty(p)$.

By the definition of generalized Bernoulli numbers, we have
$$
  \frac p2\sum_{k=1}^{(p-1)/2}B_2\left(\frac{k(p-1)}p\right)
 =\frac14B_{2,\chi_0}.
$$
We now determine $\det M$, which turns out to be essentially the
product of generalized Bernoulli numbers appearing in Theorem A.

\begin{Lemma} \label{lemma: Bernoulli matrix} Let $N\ge 4$ be an
  integer, $n=\phi(N)/2$, and
$$
  S=\{a_i:~i=1,\ldots,n,~1\le a_i\le N/2,~(a_i,N)=1\}.
$$
For an integer $b$ relatively prime to $N$, denote by $b^{-1}$ its
multiplicative inverse modulo $N$. Let $M$ be the $n\times n$ matrix
whose $(i,j)$-entry is $NB_2(a_ia_j^{-1}/N)/2$.
Then we have
$$
  \det M=\prod_\chi\frac14B_{2,\chi},
$$
where $\chi$ runs over all even characters modulo $N$.
\end{Lemma}

\begin{proof} The proof is standard. We let $V$ be the vector space
  over $\C$ of all $\C$-valued functions $f$ on $(\Z/N\Z)^\times$
  satisfying $f(a)=f(-a)$ for all $a\in(\Z/N\Z)^\times$. There are two
  standard bases for $V$. One is $\{\delta_a:~a\in S\}$, where
  $$
    \delta_a(x)=\begin{cases}1,&\text{if }x=\pm a, \\
    0,&\text{else}, \end{cases}
  $$
  and the other is $\{\text{\,all even Dirichlet characters modulo }N\}$.
  For $a\in(\Z/N\Z)^\times$, define $T_a:V\to V$ by $T_af(x)=f(ax)$.
  Then $T_a$ is a linear operator on $V$. Consider
  $$
    T=\sum_{a\in S}\frac N2B_2(a/N)T_a.
  $$
  We have
  $$
    T\delta_b(x)=\sum_{a\in S}\frac N2B_2(a/N)\delta_b(ax)
   =\sum_{a\in S}\frac N2B_2(a/N)\delta_{a^{-1}b}(x),
  $$
  and hence
  $$
    T\delta_b=\sum_{c\in S}\frac N2B_2(bc^{-1}/N)\delta_c. 
  $$
  We find that the matrix of $T$ with respect to the first basis is
  $M$.

  On the other hand, we also have, for an even Dirichlet character
  $\chi$ modulo $N$,
  $$
    T\chi(x)=\sum_{a\in S}\frac N2B_2(a/N)\chi(ax)
   =\sum_{a\in S}\frac N2B_2(a/N)\chi(a)\chi(x)=\frac14B_{2,\chi}
    \chi(x).
  $$
  From this we see that the matrix of $T$ with respect to the second
  basis is diagonal and its determinant is the product
  $$
    \prod_\chi\frac14B_{2,\chi}
  $$
  of eigenvalues, which equals to the determinant of $M$. This
  proves the lemma.
\end{proof}

In summary, in the case $N=p$ is an odd prime, if
$e_{i,j}$, $i=1,\ldots,(p-3)/2$, $j=1,\ldots,(p-1)/2$, are integers
such that
$$
  \sum_{j=1}^{(p-1)/2}e_{i,j}=0, \qquad
  \sum_{j=1}^{(p-1)/2}j^2e_{i,j}\equiv 0\mod p,
$$
and the square matrix $U=(u_{ij})$ of size $(p-1)/2$ defined by
\eqref{equation: temp u} has determinant $p$, then according to
Theorem A, $f_i=\prod_j E_j^{e_{i,j}}$, $i=1,\ldots,(p-3)/2$, form a
basis for $\FF_1^\infty(p)$.

For prime power cases $N=p^n$, the basic idea is similar. We pick a
generator $a$ of $(\Z/p^n\Z)^\times/\pm 1$ and form a
$\phi(N)/2\times\phi(N)/2$ matrix $M$ with the $(i,j)$-entry being the
order of $E_{a^{i-1}}$ at $a^{j-1}/N$. Then we try to find another
$\phi(N)/2\times\phi(N)/2$ matrix $U$ such that
\begin{enumerate}
\item the first $\phi(N)/2-1$ rows of $UM$ has the interpretation as
  the orders of some functions in $\FF_1^\infty(N)$,
\item the last row of $U$ is $(0,\ldots,0,1)$,
\item the determinant of $U$ equals to $p^{L(p)}$.
\end{enumerate}
However, unlike the prime cases, the functions $E_g$ with $(g,N)=1$
will not be sufficient to generate the whole group $\FF_1^\infty(N)$
and it is necessary to use functions from lower level, i.e., functions
of the form
$E^{(p^\ell)}_g(p^{n-\ell}\tau)=E^{(N)}_{gp^{n-\ell}}(\tau)$ for some
$\ell<n$. To record the orders of such a function at cusps, we will
invoke the distribution relation \eqref{equation: Bernoulli relation}
in Lemma \ref{lemma: Bernoulli relation}. We leave the details to
Section \ref{subsection: prime powers}.
\end{subsection}

\begin{subsection}{Method for non-prime power cases} In this section,
  we explain our idea for non-prime power cases.

  In theory, it is still possible to use the same method as the prime
  power cases, but the argument will become extremely tedious. Thus,
  instead of using Theorem A and the linear algebra argument, we use
  the following characterization of $\FF_1^\infty(N)$ of Yu \cite{Yu}.

\begin{TheoremB}[{Yu, \cite[Lemma 2.1 and Theorem 4]{Yu}}] Let $N$ be a
  positive integer having at least two distinct prime divisors. Then,
  up to a scalar, $f(\tau)$ belongs to $\FF_1^\infty(N)$ if and only
  if $f(\tau)=\prod_g E_g^{e_g}$ is a product of $E_g$ satisfying
  \begin{equation} \label{condition: DO}
    \sum_{g\in\O_{a,p}}e_g=0
  \end{equation}
  for each prime divisor $p$ of $N$ and each $a\in\Z/N\Z$.
\end{TheoremB}

  Our idea is perhaps best explained by giving an example.

\medskip

\noindent{\bf Example.} Consider $N=21$. Assume that
  $f(\tau)=\prod_{g=1}^{10}E_g^{e_g}$ is a modular unit in
  $\FF_1^\infty(21)$. Then the orbit condition \eqref{condition: DO}
  gives
  $$
    e_7=0, \quad e_3=-e_4-e_{10}, \quad e_6=-e_1-e_8, \quad e_9=-e_2-e_5,
  $$
  and
  $$
    e_1+e_2+e_4+e_5+e_8+e_{10}=0.
  $$
  Then we have
  $$
  f=\left(\frac{E_1}{E_6}\right)^{e_1}\left(\frac{E_2}{E_9}\right)^{e_2}
    \left(\frac{E_4}{E_3}\right)^{e_4}\left(\frac{E_5}{E_9}\right)^{e_5}
    \left(\frac{E_8}{E_6}\right)^{e_8}\left(\frac{E_{10}}{E_3}\right)^{e_{10}},
  $$
  subject to the condition $e_1+e_2+e_4+e_5+e_8+e_{10}=0$. Thus, if we
  let $F_i$, $i=1,2,4,5,8,10$, denote the $6$ quotients in the last
  expression, then $F_1/F_2$, $F_2/F_4$, $F_4/F_5$, $F_5/F_8$,
  $F_8/F_{10}$ will generate $\FF_1^\infty(21)$.
\medskip

  The above example shows that for a squarefree composite integer $N$,
  we may regard any $\phi(N)/2-1$ exponents $e_g$ from the set
  $\{e_g:~1\le g\le N/2,~(g,N)=1\}$ as ``free variables'' and express
  the rest of $e_g$ in terms of these free variables. This gives a
  basis for $\FF_1^\infty(N)$.

  When $N$ is not squarefree, the situation is much more complicated
  as there are relations among $\{e_g:~1\le g\le N/2,~(g,N)=1\}$ other
  than $\sum_{(g,N)=1}e_g=0$. For instance, when $N=63$, the orbit
  conditions include $e_1+e_{22}+e_{20}=0$, $e_2+e_{23}+e_{19}=0$, and
  so on. (The situation is reminiscent of the case of cyclotomic units
  where non-trivial relations exist among the units
  $1-e^{2\pi ik/N}$.) Then, again, modular units from modular curves
  of lower levels are needed to obtain a basis for $\FF_1^\infty(N)$.
  We leave the details to Section \ref{section: non-prime power}.

%
%
\end{subsection}
\end{section}

\begin{section}{Prime power cases} \label{section: prime powers}

\begin{subsection}{Prime cases} \label{subsection: primes} In this
  section, we consider the simplest case when the level is a prime.

\begin{Theorem} \label{theorem: prime} Let $N=p$ be an odd prime
  greater than $3$. Let $a$ be a generator of the cyclic group
  $(\Z/p\Z)^\times/\pm1$ and $b$ be its multiplicative inverse modulo
  $p$. Let $n=(p-1)/2$. Then the functions
  $$
    \frac{E_{a^{i-1}}E_{a^{i+1}}^{b^2}}{E_{a^i}^{1+b^2}},\quad
    i=1,\ldots,n-2,\quad\text{ and}\quad \frac{E_{b^2}^p}{E_b^p}
  $$
  generate $\FF_1^\infty(p)$ modulo $\C^\times$.
\end{Theorem}

\begin{proof} Let $p$, $a$, $b$, and $n$ be given as in the
  statement. There are $n$ essentially distinct $E_g$, and there are
  $n$ cusps $k/p$, $k=1,\ldots,n$, of $X_1(p)$ that are lying over
  $\infty$ of $X_0(p)$. By Proposition \ref{proposition: behavior of
  Eg}, the order of $E_g$ at $k/p$ is $pB_2(gk/p)/2$. We form an
  $n\times n$ matrix $M=(M_{ij})$ by setting
  $$
    M_{ij}=\frac p2B_2(a^{i+j-2}/p)
  $$
  to be the order of $E_{a^{i-1}}$ at $a^{j-1}/p$. Set
  \begin{equation} \label{equation: U1}
    U_1=\begin{pmatrix}1&-b^2& 0 & \cdots & \cdots & \cdots\\
      0& 1 & -b^2 & \cdots & \cdots & \cdots\\
      \vdots & \vdots & & & \vdots & \vdots \\
      \cdots & \cdots & \cdots & 0 & 1 & -b^2 \\
      \cdots & \cdots & \cdots & 0 & 0 & 1-b^2
    \end{pmatrix}.
  \end{equation}
  In other words, $U_1$ has $-b^2$ on the superdiagonal and $1$ on the
  diagonal, except for the last one, which has $1-b^2$. Note that for
  $i=1,\ldots,n-1$, the $i$th row of the matrix $U_1M$ now records the
  orders of $E_{a^{i-1}}/E_{a^i}^{b^2}$ at cusps.

  Furthermore, set
  \begin{equation} \label{equation: U2}
    U_2=\begin{pmatrix}1&-1& 0 & \cdots & \cdots & \cdots \\
      0 & 1 & -1 & \cdots & \cdots & \cdots \\
      \vdots & \vdots & & & \vdots & \vdots \\
      \cdots & \cdots & \cdots & 1 & -1 & 0 \\
      \cdots & \cdots & \cdots & 0 & p & -p \\
      \cdots & \cdots & \cdots & 0 & 0 & 1
    \end{pmatrix}
  \end{equation}
  The matrix $U_2$ has $-1$ on the superdiagonal and $1$ on the
  diagonal, except for the $(n-1)$-st row, which has $-p$ and $p$,
  respectively. Then the first $n-2$ rows of the matrix $U_2U_1M$
  describe the orders of
  $f_i=E_{a^{i-1}}E_{a^{i+1}}^{b^2}/E_{a^i}^{1+b^2}$ at cusps. Also,
  the $(n-1)$-st row gives the orders of $f_{n-1}=E_{b^2}^p/E_b^p$ at
  cusps. By Proposition \ref{proposition: Gamma1 invariant}, these
  functions are all modular on $\Gamma_1(p)$. Moreover, by Proposition
  \ref{proposition: behavior of Eg}, $f_i$ has no poles nor zeros at
  cusps that are not of the form $k/p$. Thus, the functions $f_i$
  belong to the group $\FF_1^\infty(p)$. We will show that these
  functions are a basis of $\FF_1^\infty(p)$ modulo $\C^\times$, or
  equivalently, that the divisors of these functions $f_i$ form a
  $\Z$-basis for the additive group $\div\FF^\infty_1(p)$.

  It is obvious that the divisor group $\DD^\infty_1(p)$ is
  generated by $(1/p)-(a/p)$, $(a/p)-(a^2/p)$, $\cdots$,
  $(b^2/p)-(b/p)$. Thus, according to Lemma \ref{lemma: lattice
  index}, the index of the subgroup generated by the
  divisors of $f_i$ in the group $\DD^\infty_1(p)$ is the absolute
  value of $\det(U_2U_1M)$ divided by the sum of the entries in the
  last row of $U_2U_1M$. It remains to show that it has the correct
  value as given in \eqref{equation: Yu's formula}.

  By Lemma \ref{lemma: Bernoulli matrix}, the determinant of $M$ is,
  up to $\pm1$ sign,
  $$
    \det M=\prod_{\chi}\frac14B_{2,\chi}
  $$
  (Note that the matrix $M$ here differs from the one in Lemma
  \ref{lemma: Bernoulli matrix} by multiplication by permutation
  matrices on the two sides.) Also, the determinants of $U_1$ and $U_2$
  are $1-b^2$ and $p$, respectively. Now the last row of $U_2U_1$ is
  $(0,\ldots,0,1-b^2)$. If follows that the sum of the entries in the
  last row of $U_2U_1M$ is equal to
  $$
    (1-b^2)\sum_{k=1}^{(p-1)/2}\frac p2B_2(k/p)
   =\frac{1-b^2}4B_{2,\chi_0}.
  $$
  and the index is equal to the absolute value of
  \begin{equation*}
  \begin{split}
    \frac 4{(1-b^2)B_{2,\chi_0}}\det(U_2U_1M)
   =\frac 4{(1-b^2)B_{2,\chi_0}}\cdot p
    \cdot(1-b^2)\prod_\chi\frac14B_{2,\chi}
   =p\prod_{\chi\neq\chi_0}\frac14B_{2,\chi},
  \end{split}
  \end{equation*}
  which is indeed the index of $\div\FF_1^\infty(p)$ in
  $\DD^\infty_1(p)$. In other words, the functions $f_i$ form a basis
  for $\FF_1^\infty(p)$ modulo $\C^\times$. This completes the proof
  of the theorem.
\end{proof}

We now give an example demonstrating our idea.
\medskip

\noindent{\bf Example.} Let $N=13$. We choose the generator $a$ of the
group $(\Z/13\Z)^\times$ to be $a=7$. Then the multiplicative inverse
of $a$ modulo $13$ is $b=2$. The cusps of $X_1(13)$ lying over
$\infty$ of $X_0(13)$ are $i/13$, $i=1,\ldots,6$. We denote these
cusps by $P_i=a^{i-1}/13$, $i=1,\ldots,6$. Then the matrix $M$ is
$$
  M=(13B_2(7^{i+j-2}/13)/2)_{ij}
   =\frac1{156}\begin{pmatrix}
  97 & -83 & -11 & -71 & -47 & 37 \\
 -83 & -11 & -71 & -47 &  37 & 97 \\
 -11 & -71 & -47 &  37 &  97 & -83 \\
 -71 & -47 &  37 &  97 & -83 & -11 \\
 -47 &  37 &  97 & -83 & -11 & -71 \\
  37 &  97 & -83 & -11 & -71 & -47 \end{pmatrix}.
$$
With $U_1$ and $U_2$ given by \eqref{equation: U1} and
\eqref{equation: U2}, we find
$$
  U_2U_1M=\begin{pmatrix}
  3&-2&1&2&1&-5\\-2&1&2&1&-5&3\\1&2&1&-5&3&-2\\
  2&1&-5&3&-2&1\\-7&-5&15&-6&5&-2\\
  c_1&c_2&c_3&c_4&c_5&c_6\end{pmatrix},
$$
where $(c_1,\ldots,c_6)$ is $-3$ times the last row of $M$.
The first $5$ rows represent the orders of $f_1=E_1E_3^4/E_6^5$,
$f_2=E_6E_5^4/E_3^5$, $f_3=E_3E_4^4/E_5^5$, $f_4=E_5E_2^4/E_4^5$, and
$E_4^{13}/E_2^{13}$ at the cusps $P_j$, respectively. Since the
determinant of $U_2U_1M$ is $57/2$ and the sum of $c_i$ is $3/2$,
according to Lemma \ref{lemma: lattice index}, we find the index of
the subgroup generated by the divisors of $f_i$ in $\DD^\infty_1(13)$
is $19$, which agrees with the divisor class number obtained from
\eqref{equation: Yu's formula}. In other words, $f_i$ generate
$\FF_1^\infty(13)$.

Now observe that the divisor class group
$\DD_1^\infty(13)/\div\FF_1(13)$ is cyclic. Thus, there is an integer
$m$ with $0<m<19$ such that $m(P_1)-m(P_2)$ and $(P_2)-(P_3)$ are in
the same class, i.e., $m(P_1)-(m+1)(P_2)+(P_3)$ is a principal
divisor. This integer $m$ has the property that the equation
$$
  (c_1,c_2,c_3,c_4,c_5,c_6)U_3U_2U_1M
 =(m,-m-1,1,0,0,0)
$$
has an integer solution. We find that this occurs when $m=8$ with
$$
  (c_1,\ldots,c_6)=(9,40,167,675,208,0).
$$
This integer $m$ also satisfies $m(P_2)-m(P_3)\sim(P_3)-(P_4)$,
$m(P_3)-m(P_4)\sim(P_4)-(P_5)$, and so on. This is because if
$f(\tau)$ is a modular function on $\Gamma_1(13)$ such that
$$
  \div f=m(P_1)-(m+1)(P_2)+(P_3),
$$
then
\begin{equation*}
\begin{split}
  \div f\big|\begin{pmatrix}6&-2\\13&-1\end{pmatrix}
&=m(P_2)-(m+1)(P_3)+(P_4), \\
  \div f\big|\begin{pmatrix}6&-2\\13&-1\end{pmatrix}^2
&=m(P_3)-(m+1)(P_4)+(P_5),
\end{split}
\end{equation*}
and so on. From these informations, we see that the divisor
$$
  \sum_{i=1}^6d_i(P_i)
$$ 
and 
\begin{equation*}
\begin{split}
&\{d_1+8(d_1+d_2)+8^2(d_1+d_2+d_3)+\cdots+8^4(d_1+d_2+d_3+d_4+d_5)\} \\
&\qquad\cdot  ((P_1)-(P_2))+(d_1+d_2+\cdots+d_6)(P_6)
\end{split}
\end{equation*}
are in the same divisor class. In particular, it is principal if and
only if
$$
  d_1+d_2+\cdots+d_6=0, \qquad
  7d_1+6d_2+17d_3+10d_4+11d_5\equiv 0\mod 19.
$$
\end{subsection}

\begin{subsection}{Prime power cases} \label{subsection: prime powers}
In this section we deal with the cases where $N=p^k$ is a prime power.
For the ease of exposition, odd prime power cases and even prime power
cases are stated as two theorems, even though the proofs are very
similar.

We first describe two constructions of modular functions
belonging to $\FF_1^\infty(p^k)$.

\begin{Lemma} \label{lemma: method 1} Let $p$ be a prime, and $k$ be
  an integer greater than $1$. Suppose that $g$ and $e_g$ are integers
  satisfying $p\nmid g$ and
  $$
    \sum_g g^2e_g\equiv 0\mod p
  $$
  Then
  $$
    \prod_g\left(\frac{E_g}{E_{g(1+mp^{k-1})}}\right)^{e_g}
  $$
  is a modular function in $\FF_1^\infty(p^k)$ for all integers $m$.
\end{Lemma}

\begin{proof} It is easy to see from Proposition \ref{proposition:
  Gamma1 invariant} that the functions concerned are all modular on
  $\Gamma_1(p^k)$. These functions also satisfy the orbit condition in
  Lemma \ref{lemma: orbit condition}. Thus, they are contained in
  $\FF_1^\infty(p^k)$.
%
\end{proof}

The second method uses functions from lower levels.

\begin{Lemma} \label{lemma: level lowering} Assume that $N=pM$.
  If $e_g$ are integers such that
  $$
    \sum_g e_g\equiv 0\mod 12, \qquad
    \sum_g ge_g\equiv 0\mod 2,
  $$
  and
  $$
    \sum_g g^2e_g\equiv 0\mod
    \begin{cases}2M, &\text{if }p\nmid M, \\
    2M/p, &\text{if }p|M, \end{cases}
  $$
  then $\prod_g E_g^{(M)}(p\tau)^{e_g}$ is a modular function on
  $\Gamma_1(N)$. Moreover, if $M$ is odd, then the conditions can
  be relaxed to
  $$
    \sum_g e_g\equiv 0\mod 12, \qquad
    \sum_g g^2e_g\equiv 0\mod
    \begin{cases}M, &\text{if }p\nmid M, \\
    M/p, &\text{if }p|M. \end{cases}
  $$
\end{Lemma}

\begin{proof} Let
  $\sigma=\left(\begin{smallmatrix}a&b\\c&d\end{smallmatrix}\right)
   \in\Gamma_1(N)$.
  We have
  \begin{equation} \label{equation: p sigma}
    p\sigma\tau
   =p\frac{a\tau+b}{c\tau+d}=\frac{a(p\tau)+bp}{(c/p)(p\tau)+d}
   =\begin{pmatrix}a&bp\\ c/p&d\end{pmatrix}(p\tau).
  \end{equation}
  Then by \eqref{Eg formula} of Proposition \ref{proposition: Eg
  formulas},
  $$
    E_g^{(M)}(p\sigma\tau)=\epsilon(a,bpM,c/pM,d)
    e^{\pi i(g^2abp/M-gbp)}E_{ag}^{(M)}(p\tau).
  $$
  Then the assumptions $\sum_g e_g\equiv 0\mod 12$ and $\sum_g
  ge_g\equiv 0\mod 2$ imply that
  $$
    \prod E_g^{(M)}(p\sigma\tau)^{e_g}=\exp\left\{\pi iabp
    \sum g^2e_g/M\right\}\prod E^{(M)}_{ag}(p\tau)^{e_g}.
  $$
  When $p\nmid M$, the condition $\sum_g g^2e_g\equiv 0\mod 2M$ ensures
  that the exponential factor is equal to $1$. When $p|M$, the
  condition $\sum_g g^2e_g\equiv 0\mod 2M/p$ will suffice. In either
  case, we have
  $$
    \prod E_g^{(M)}(p\sigma\tau)^{e_g}=
    \prod E_{ag}^{(M)}(p\tau)^{e_g}.
  $$
  Finally, equality \eqref{shifting for Eg} in Proposition
  \ref{proposition: Eg formulas} and the assumption $\sum ge_g\equiv
  0\mod 2$ show that
  $$
    \prod E_g^{(M)}(p\sigma\tau)^{e_g}=
    \prod(-1)^{ge_g(a-1)/N}E_g^{(M)}(p\tau)^{e_g}=
    \prod E_g^{(M)}(p\tau)^{e_g}.
  $$

  When $M$ is odd, since $E_g^{(M)}=E_{M-g}^{(M)}$, we may assume that
  all $g$ are even so that $\sum ge_g\equiv 0,~\sum g^2e_g\equiv 0\mod
  2$ are always satisfied. Also, $g^2\equiv (M-g)^2\mod M$. Therefore,
  the conditions can be reduced to $\sum e_g\equiv 0\mod 12$ and $\sum
  g^2e_g\equiv 0\mod M$ when $M$ is odd. This completes the proof.
\end{proof}

\begin{Lemma} \label{lemma: level lowering 2} Let $p$ a prime. Assume
  that $N=pM$. If $f(\tau)$ is a modular function on $X_1(M)$, then
  $f(p\tau)$ is a modular function on $X_1(N)$. Furthermore, if $p$
  also divides $M$ and $f(\tau)$ belongs to the group
  $\FF_1^\infty(M)$, then the function $f(p\tau)$ belongs to the group
  $\FF_1^\infty(N)$.
\end{Lemma}

\begin{proof} Assume that $f(\tau)$ is modular on $\Gamma_1(M)$. Given
  $\sigma=\left(\begin{smallmatrix} a&b\\
  c&d\end{smallmatrix}\right)\in\Gamma_1(N)$, we have, by
  \eqref{equation: p sigma},
  $$
    p\sigma\tau=\begin{pmatrix}a&bp\\ c/p&d\end{pmatrix}(p\tau).
  $$
  Since $f(\tau)$ is assumed to be modular on $\Gamma_1(M)$, we have
  $f(p\sigma\tau)=f(p\tau)$. That is, $f(p\tau)$ is modular on
  $\Gamma_1(N)$.

  Now assume that $p|M$ and $f(\tau)\in\FF_1^\infty(M)$. The
  assumption that $f(\tau)$ has no zeros nor poles in $\H$ implies
  that $f(p\tau)$ has the same property. Now we check that the poles
  and zeros of $f(p\tau)$ occurs only at cusps lying over $\infty$ of
  $X_0(N)$.

  In general, the cusps of $X_1(N)$ takes the form $a/c$ with $c|N$
  and $(a,c)=1$. Choose integers $b$ and $d$ such that
  $$
    \sigma=\begin{pmatrix}a&b\\c&d\end{pmatrix}\in SL(2,\Z).
  $$
  Then the order of a modular function $g(\tau)$ on $X_1(N)$ at $a/c$ is
  determined by the Fourier expansion of $g(\sigma\tau)$. In
  particular, $g(\tau)$ has no pole nor zero at $a/c$ if and only if
  the Fourier expansion of $g(\sigma\tau)$ starts from a non-vanishing
  constant term.

  Now $a/c$ does not lie over $\infty$ of $X_0(N)$ if and only $c$ is
  a proper divisor of $N$. We will show that $f(p\tau)$ has no
  poles nor zeros at such points. This amounts to proving the
  assertion that $\lim_{\tau\to\infty}f(p\sigma\tau)$ is finite and
  non-vanishing for such $a/c$. We consider two cases $p|c$ and $p\nmid
  c$ separately.

  When $p|c$, we have $\lim_{\tau\to\infty}f(p\sigma\tau)=f(a/(c/p))$.
  Since the denominator $c/p$ is a proper divisor of $M$, by
  assumption that $f(\tau)\in\FF_1^\infty(M)$,
  $\lim_{\tau\to\infty}f(p\sigma\tau)$ is finite and non-vanishing.

  When $p\nmid c$, we have $\lim_{\tau\to\infty}p\sigma\tau=pa/c$. By the
  assumption that $p|M$, the denominator $c$ is a proper
  divisor of $M$. Thus, we conclude again that $f(p\tau)$ has no poles
  nor zeros at $a/c$. This completes the proof.
\end{proof}

Combining the above two lemmas, we obtain a simple construction of
modular functions that are in $\FF_1^\infty(p^k)$.

\begin{Corollary} \label{corollary: method 2} Let $p$ be a prime and
  $k\ge 2$ be a positive integer. Then 
  $$
    E^{(p^{\ell})}_g(p^{k-\ell}\tau)/
    E^{(p^{\ell})}_{g+mp^{\ell-1}}(p^{k-\ell}\tau)
  $$
  are all modular functions contained in $\FF_1^\infty(p^k)$ for all
  $\ell=1,\ldots,k-1$ and all $g$ and $m$ satisfying
  $g,~g+mp^{\ell-1}\not\equiv 0\mod p^{\ell}$.
\end{Corollary}

\begin{proof} First of all, Lemma \ref{lemma: level lowering} shows
  that
  $f(\tau)=E^{(p^\ell)}_g(p\tau)/E^{(p^\ell)}_{g+mp^{\ell-1}}(p\tau)$
  is a modular function on $\Gamma_1(p^{\ell+1})$. Then the first part
  of Lemma \ref{lemma: level lowering 2} implies that
  $f(p^{k-\ell-1}\tau)$ is modular on $\Gamma_1(p^k)$ for all
  $k>\ell$. We now prove that it has poles and zeros only at cusps in
  $C_1^\infty(p^k)$.

  Lemma \ref{lemma: method 1} shows that
  $f(\tau/p)^p=E^{(p^\ell)}_g(\tau)^p/E^{(p^\ell)}_{g+mp^{\ell-1}}(\tau)^p$
  is in $\FF_1^\infty(p^\ell)$. Then Lemma \ref{lemma: level lowering 2}
  implies that $f(p^{k-\ell-1}\tau)^p$ has poles and zeros only at
  cusps in $C_1^\infty(p^k)$, and so is $f(p^{k-\ell-1}\tau)$. We
  conclude that $f(p^{k-\ell-1}\tau)$ is in $\FF_1^\infty(p^k)$, as
  claimed in the statement of the lemma.
\end{proof}

With the above lemmas we can now determine a basis for
$\FF_1^\infty(p^k)$ for odd primes $p$ and integers $k\ge 2$.

\begin{Theorem} \label{theorem: prime power} Let $k>1$ and $N=p^k$ be
  an odd prime power. For a positive integer $\ell$, we set
  $\phi_\ell=\phi(p^\ell)/2$. Let $a$ be a generator of the cyclic
  group $(\Z/p^k\Z)^\times/\pm 1$ and $b$ be its multiplicative
  inverse modulo $p$. Then a basis for $\FF_1^\infty(p^k)$ modulo
  $\C^\times$ is given by
$$
  \begin{cases}
  f_i=\frac{E_{a^{i-1}}E_{a^{i+\phi_{k-1}}}^{b^2}}
  {E_{a^{i+\phi_{k-1}-1}}E_{a^i}^{b^2}},
  &i=1,\ldots,\phi_k-\phi_{k-1}-1, \\
  f_i=\frac{E_{a^{i-1}}^p}{E_{a^{i+\phi_{k-1}-1}}^p},
  &i=\phi_k-\phi_{k-1}, \\
  f_i=\frac{E^{(p^{k-1})}_{a^{i-1}}(p\tau)}
  {E^{(p^{k-1})}_{a^{i+\phi_{k-2}-1}}(p\tau)},
  &i=\phi_k-\phi_{k-1}+1,\ldots,\phi_k-\phi_{k-2}, \\
  \qquad \vdots & \qquad \vdots \\
  f_i=\frac{E^{(p^2)}_{a^{i-1}}(p^{k-2}\tau)}
    {E^{(p^2)}_{a^{i+\phi_1-1}}(p^{k-2}\tau)},
  &i=\phi_k-\phi_2+1,\ldots,\phi_k-\phi_1, \\
  f_i=\frac{E^{(p)}_{a^{i-1}}(p^{k-1}\tau)}
  {E^{(p)}_{a^i}(p^{k-1}\tau)}, &i=\phi_k-\phi_1+1,\ldots,\phi_k-1.
  \end{cases}
$$
\end{Theorem}

\begin{proof} Form a $\phi_k\times\phi_k$ matrix $M=(M_{ij})$, where
$$
  M_{ij}=\frac{p^k}2B_2\left(\frac{a^{i+j-2}}{p^k}\right).
$$
In other words, $M_{ij}$ is the order of $E_{a^{i-1}}$ at
$a^{j-1}/p^k$. For $\ell=2,\ldots,k$, define a
$\phi_\ell\times\phi_\ell$ matrix $V_\ell$ by
\begin{equation} \label{equation: Vell}
  V_\ell=\begin{pmatrix}
  I & -I & 0 & \cdots & \cdots & \cdots \\
  0 & I & -I & \cdots & \cdots & \cdots \\
  \vdots & \vdots & & & \vdots & \vdots \\
  \cdots & \cdots & \cdots & I & -I & 0 \\
  \cdots & \cdots & \cdots & 0 & I & -I \\
  pI & pI & \cdots & \cdots & pI & pI \end{pmatrix}.
\end{equation}
Here the matrix consists of $p^2$ blocks, each of which is of
$\phi_{\ell-1}\times\phi_{\ell-1}$ size, and $I$ is the
identity matrix, while $0$ is the zero matrix. 
Define also $\phi_k\times\phi_k$ matrices $U_\ell$, $\ell=1,\ldots,k$, by
\begin{equation} \label{equation: Uk}
  U_k=V_k
\end{equation}
and for $\ell=2,\ldots,k-1$,
\begin{equation} \label{equation: Uell}
  U_\ell=\begin{pmatrix}
  I & 0 & 0 & \cdots & \cdots & \cdots \\
  0 & I & 0 & \cdots & \cdots & \cdots \\
  \vdots & \vdots & & & \vdots & \vdots \\
  \cdots & \cdots & \cdots & I & 0 & 0 \\
  \cdots & \cdots & \cdots & 0 & I & 0 \\
  0 & 0 & \cdots & \cdots & 0 & V_\ell \end{pmatrix}.
\end{equation}
In other words, $U_\ell$ is obtained by replacing the lower right
$\phi_\ell\times\phi_\ell$ block of an $\phi_k\times\phi_k$
identity matrix by $V_\ell$. Also, define $U_1$ to be
$$
  U_1=\begin{pmatrix}
  I & 0 & 0 & \cdots & \cdots & \cdots \\
  0  & 1 & -1 & 0 & 0 & \cdots \\
  0  & 0 & 1 & -1 & 0 & \cdots \\
  \vdots & \vdots & & & \vdots & \vdots \\
  0 & 0 & \cdots & \cdots & 1 & -1 \\
  0 & 0 & \cdots & \cdots & 0 & 1\end{pmatrix},
$$
where the identity block in the upper left corner is of size
$\phi_k-\phi_1$. Finally, let $b$ be the multiplicative
inverse of $a$ modulo $p$ and define $U_k'$ by
\begin{equation} \label{equation: Uk'}
  U_k'=\begin{pmatrix}
  1 & -b^2 & 0 & \cdots & \cdots & \cdots \\
  0 & 1 & -b^2 & \cdots & \cdots & \cdots \\
  \vdots & \vdots & & & \vdots & \vdots \\
  \cdots & \cdots & \cdots & 1 & -b^2 & 0 \\
  \cdots & \cdots & \cdots & 0 & p & 0 \\
  0 & 0 & \cdots & \cdots & 0 & I \end{pmatrix}
\end{equation}
where the $I$ in the lower right corner is the identity matrix of size
$\phi_{k-1}$, and the diagonals are all $1$ except for the row
just above $I$. Also, the $(i,i+1)$-entries are $-b^2$ for
$i=1,\ldots,\phi_k-\phi_{k-1}-1$. Now let us consider the
effect of the multiplication of $M$ by 
$U_1U_2\ldots U_{k-1}U_k'U_k$ on the left. We will show that the first
$\phi_k-1$ rows of the resulting matrix will be a basis for the lattice
corresponding to $\div\FF_1^\infty(p^k)$.

Clearly, the first $\phi_k-\phi_{k-1}$ rows of $U_kM$ record
the orders of the functions
$E^{(p^k)}_{a^{i-1}}/E^{(p^k)}_{a^{i+\phi_{k-1}-1}}$. The entries in
the last few rows of $U_kM$ take the form
$$
  p\sum_{m=0}^{p-1}\frac{p^k}2B_2
  \left(\frac{a^{i+j+m\phi_{k-1}-2}}{p^k}\right).
$$
Now observe that as $m$ runs through $0$ to $p-1$, $a^{m\phi_{k-1}}$
modulo $p^k$ goes through exactly once elements of
$\{1+\ell p^{k-1}:\ell=0,\ldots,p-1\}$ modulo $\pm 1$. (Note that
since $a$ is a generator of $(\Z/p^k\Z)^\times$,
$a^{(p-1)/2}=-1+\ell p$ for some $\ell$ not divisible by $p$. Then
$a^{\phi_{k-1}}=(-1+\ell p)^{p^{k-2}}\equiv -1+\ell p^{k-1}\mod p^k$.)
Then by the relation between generalized Bernoulli numbers given in
Lemma \ref{lemma: Bernoulli relation}, the above sum is equal to
$$
  p\cdot\frac{p^{k-1}}2B_2\left(\frac{a^{i+j-2}}{p^{k-1}}\right),
$$
which we can interpret as the order of $E^{(p^{k-1})}_{a^{i-1}}(p\tau)$ at
cusps $a^{j-1}/p^k$.

Then the first $\phi_k-\phi_{k-1}-1$ rows of $U_k'U_kM$ will be the
orders of
$$
  f_i=\frac{E_{a^{i-1}}E_{a^{i+\phi_{k-1}}}^{b^2}}
  {E_{a^{i+\phi_{k-1}-1}}E_{a^i}^{b^2}},
  \qquad i=1,\ldots,\phi_k-\phi_{k-1}-1,
$$
at cusps. Also, the $(\phi_k-\phi_{k-1})$-th row of
$U_k'U_kM$ corresponds to the function
$$
  f_{\phi_k-\phi_{k-1}}=
  \frac{E_{a^{\phi_k-\phi_{k-1}-1}}^p}{E_{a^{\phi_k-1}}^p}.
$$
Since $a^{\phi_{k-1}}=a^{p^{k-2}(p-1)/2}\equiv\pm1\mod p^{k-1}$, by Lemma
\ref{lemma: method 1}, these functions $f_i$ are all in
$\FF_1^\infty(p^k)$. Summarizing, we find that the rows of $U_k'U_kM$
record the orders of the functions
$$
  f_1,\ldots,f_{\phi_k-\phi_{k-1}},E^{(p^{k-1})}_1(p\tau),
  E^{(p^{k-1})}_{a^{\phi_k-\phi_{k-1}}}(p\tau),\ldots,
  E^{(p^{k-1})}_{a^{\phi_k-1}}(p\tau)
$$
at cusps in $C_1^\infty(p^k)$.

Now consider $U_{k-1}U_k'U_kM$. The multiplication of $U_{k-1}$ on the
left leaves the first $\phi_k-\phi_{k-1}$ rows invariant. The next
$\phi_{k-1}-\phi_{k-2}$ rows now record the order of
$$
  f_i=\frac{E^{(p^{k-1})}_{a^{i-1}}(p\tau)}
  {E^{(p^{k-1})}_{a^{i+\phi_{k-2}-1}}(p\tau)}, \qquad
  i=\phi_k-\phi_{k-1}+1,\ldots,\phi_k-\phi_{k-2},
$$
which, by Corollary \ref{corollary: method 2}, are all modular
functions contained in $\FF_1^\infty(p^k)$. Also, the last $\phi_{k-2}$
rows of $U_{k-1}U_k'U_kM$ now have
$$
  p^2\sum_{m=0}^{p-1}\frac{p^{k-1}}2B_2\left(
  \frac{a^{i+j+m\phi_{k-2}-2}}{p^{k-1}}\right)
$$
as entries. By Lemma \ref{lemma: Bernoulli relation}, the above sum is
equal to
$$
  p^2\cdot\frac{p^{k-2}}2B_2\left(\frac{a^{i+j-2}}{p^{k-2}}\right).
$$
This number is precisely the order of
$E^{(p^{k-2})}_{a^{i-1}}(p^2\tau)$ at $a^{j-1}/p^k$.

Continuing in the same way, we find that the rows of $U_1U_2\ldots
U_{k-1}U_k'U_kM$ represent the orders of the functions
$$
  \begin{cases}
  f_i=\frac{E_{a^{i-1}}E_{a^{i+\phi_{k-1}}}^{b^2}}
  {E_{a^{i+\phi_{k-1}-1}}E_{a^i}^{b^2}},
  &i=1,\ldots,\phi_k-\phi_{k-1}-1, \\
  f_i=\frac{E_{a^{i-1}}^p}{E_{a^{i+\phi_{k-1}-1}}^p},
  &i=\phi_k-\phi_{k-1}, \\
  f_i=\frac{E^{(p^{k-1})}_{a^{i-1}}(p\tau)}
  {E^{(p^{k-1})}_{a^{i+\phi_{k-2}-1}}(p\tau)},
  &i=\phi_k-\phi_{k-1}+1,\ldots,\phi_k-\phi_{k-2}, \\
  \qquad \vdots & \qquad \vdots \\
  f_i=\frac{E^{(p^2)}_{a^{i-1}}(p^{k-2}\tau)}
    {E^{(p^2)}_{a^{i+\phi_1-1}}(p^{k-2}\tau)},
  &i=\phi_k-\phi_2+1,\ldots,\phi_k-\phi_1, \\
  f_i=\frac{E^{(p)}_{a^{i-1}}(p^{k-1}\tau)}
  {E^{(p)}_{a^i}(p^{k-1}\tau)}, &i=\phi_k-\phi_1+1,\ldots,\phi_k-1, \\
  E^{(p)}_{a^{i-1}}(p^{k-1}\tau), &i=\phi_k.
  \end{cases}
$$
Except for the last one, the functions are all modular functions
belonging to $\FF_1^\infty(p^k)$. We now prove that these functions
form a basis for $\FF_1^\infty(p^k)$ modulo $\C^\times$. In view of
Lemma \ref{lemma: lattice index}, we need to show that the absolute
value of the determinant of $U_1\ldots U_{k-1}U_k'U_kM$ divided by the
sum of the entries in its last row is equal to $h_1^\infty(p^k)$.

We first determine the sum of the entries in the last row of
$U_1\ldots U_{k-1}U_k'U_kM$. Let $w$ denote the sequence
$1,0,\ldots,0$, of length $(p-1)/2$, where the first number is $1$ and
the rest are $0$. Then inductively, we can show that the last row
of $U_1\ldots U_\ell$ is $(0,\ldots,0,p^{\ell-1}w,\ldots,p^{\ell-1}w)$
for $1<\ell<k$, where there are $p^{\ell-1}$ copies of $p^{\ell-1}w$
at the end of the row with the preceding entries all being $0$. The
multiplication on the right by $U_k'$ does not change the last row of
$U_1\ldots U_{k-1}$. Then we find the last row of $U_1\ldots
U_{k-1}U_k'U_k$ is
$(p^{k-1}w,\ldots,p^{k-1}w)=(p^{k-1},0,\ldots,0,p^{k-1},\ldots)$,
where there are $(p-3)/2$ zeros between two $p^{k-1}$. Thus, the sum
of the entries in the last row of $U_1\ldots U_{k-1}U_k'U_kM$ is
$$
  p^{2k-2}\sum_{j\le p^k/2,\,p\nmid j}\frac{p^k}2B_2(j/p^k)
 =\frac{p^{2k-2}}4B_{2,\chi_0}.
$$

Now we have $\det U_1=1$, $\det U_k'=p$, and
$\det U_\ell=p^{2\phi_{\ell-1}}$ for $\ell=2,\ldots,k$. Also, by Lemma
\ref{lemma: Bernoulli matrix}, up to a $\pm1$ sign,
$$
  \det M=\prod_{\chi}\frac14B_{2,\chi}
$$
Thus, the index is equal to the absolute value of
\begin{equation*}
\begin{split}
  \frac4{p^{2k-2}B_{2,\chi_0}}\cdot
  p^{2(\phi_1+\phi_2+\cdots+\phi_{k-1})}\cdot p\cdot
  \prod_{\chi}\frac14B_{2,\chi}
 =p^{p^{k-1}-2k+2}\prod_{\chi\neq\chi_0}\frac14B_{2,\chi}.
\end{split}
\end{equation*}
This is exactly the class number given in Yu's formula. In other
words, $f_1$, $i=1,\ldots,\phi_k-1$, generate $\FF_1^\infty(p^k)$.
This completes the proof.
\end{proof}
\medskip

\noindent{\bf Example.} Consider $N=27$. A generator of
$(\Z/27\Z)^\times$ is $2$. With the notations as above, we have
$$
  M=\frac1{108}\left(\begin{smallmatrix}
  191&143&59&-61&-109&23&-97&-37&-121 \\
  143&59&-61&-109&23&-97&-37&-121&191 \\
  59&-61&-109&23&-97&-37&-121&191&143 \\
  -61&-109&23&-97&-37&-121&191&143&59 \\
  -109&23&-97&-37&-121&191&143&59&-61 \\
  23&-97&-37&-121&191&143&59&-61&-109 \\
  -97&-37&-121&191&143&59&-61&-109&23 \\
  -37&-121&191&143&59&-61&-109&23&-97 \\
  -121&191&143&59&-61&-109&23&-97&-37
  \end{smallmatrix}\right),
$$
where the $(i,j)$-entry is $27B_2(2^{i+j-2}/27)/2$, $U_1$ is equal to
the identity matrix,
$$
  U_2=\left(\begin{smallmatrix}
  1&0&0&0&0&0&0&0&0 \\
  0&1&0&0&0&0&0&0&0 \\
  0&0&1&0&0&0&0&0&0 \\
  0&0&0&1&0&0&0&0&0 \\
  0&0&0&0&1&0&0&0&0 \\
  0&0&0&0&0&1&0&0&0 \\
  0&0&0&0&0&0&1&-1&0 \\
  0&0&0&0&0&0&0&1&-1 \\
  0&0&0&0&0&0&3&3&3\end{smallmatrix}\right), \qquad
  U_3=\left(\begin{smallmatrix}
  1&0&0&-1&0&0&0&0&0 \\
  0&1&0&0&-1&0&0&0&0 \\
  0&0&1&0&0&-1&0&0&0 \\
  0&0&0&1&0&0&-1&0&0 \\
  0&0&0&0&1&0&0&-1&0 \\
  0&0&0&0&0&1&0&0&-1 \\
  3&0&0&3&0&0&3&0&0 \\
  0&3&0&0&3&0&0&3&0 \\
  0&0&3&0&0&3&0&0&3\end{smallmatrix}\right),
$$
and
$$
  U_3'=\left(\begin{smallmatrix}
  1&-1&0&0&0&0&0&0&0 \\
  0&1&-1&0&0&0&0&0&0 \\
  0&0&1&-1&0&0&0&0&0 \\
  0&0&0&1&-1&0&0&0&0 \\
  0&0&0&0&1&-1&0&0&0 \\
  0&0&0&0&0&3&0&0&0 \\
  0&0&0&0&0&0&1&0&0 \\
  0&0&0&0&0&0&0&1&0 \\
  0&0&0&0&0&0&0&0&1\end{smallmatrix}\right).
$$
Then setting $Q=U_1U_2U_3'U_3M$, we have
$$
  Q=\begin{pmatrix}
  0&2&0&1&-2&4&-1&0&-4 \\
  2&0&1&-2&4&-1&0&-4&0 \\
  0&1&-2&4&-1&0&-4&0&2 \\
  1&-2&4&-1&0&-4&0&2&0 \\
  -2&4&-1&0&-4&0&2&0&1 \\
  4&-8&-5&-5&7&7&1&1&-2\\
  1&1&-2&1&1&-2&1&1&-2 \\
  1&-2&1&1&-2&1&1&-2&1 \\
  c&c&c&c&c&c&c&c&c \end{pmatrix}, \quad c=-3/4.
$$
The first $8$ rows correspond to the divisors of the functions
$$
  \frac{E_1E_{11}}{E_2E_8},\ \frac{E_2E_5}{E_4E_{11}},\
  \frac{E_4E_{10}}{E_8E_5},\ \frac{E_8E_7}{E_{11}E_{10}},\
  \frac{E_{11}E_{13}}{E_5E_7},\ \frac{E_5^3}{E_{13}^3},\
  \frac{E^{(9)}_1(3\tau)}{E^{(9)}_2(3\tau)}, \
  \frac{E^{(9)}_2(3\tau)}{E^{(9)}_4(3\tau)}.
$$
The determinant of $Q$ is $-4252257/4$. Thus, by Lemma \ref{lemma:
  lattice index}, the index is $4252257/4\times
4/27=3^3\cdot19\cdot307$, which is indeed what one would get from
\eqref{equation: Yu's formula}. We now determine the structure of the
divisor class group $\DD_1^\infty(27)/\div\FF_1^\infty(27)$.

Let $P_i$ denote the cusps $2^{i-1}/27$. Let $Q'$ be the $8\times 9$
matrix formed by deleting the last row of $Q$. Then the Hermite normal
form for $Q'$ is
$$
  \begin{pmatrix}
  1&0&0&0&0&0&-1&13842&-13842\\0&1&0&0&0&0&0&-19882&19881\\
  0&0&1&0&0&0&0&2511&-2512\\0&0&0&1&0&0&-1&-17037&17037\\
  0&0&0&0&1&0&0&11942&-11943\\0&0&0&0&0&1&0&7047&-7048\\
  0&0&0&0&0&0&3&-22245&22242\\0&0&0&0&0&0&0&52497&-52497\end{pmatrix}.
$$
From this we see that the divisor class group is isomorphic to
$C_{52497}\times C_3$, and generated by the classes of
$$
  v_1=(P_8)-(P_9), \qquad
  v_2=(P_7)-7415(P_8)+7714(P_9).
$$
Furthermore, for a divisor $\sum_{i=1}^9d_i(P_i)$ of degree $0$, we
have
\begin{equation*}
\begin{split}
  \sum_{i=1}^9d_i(P_i)
 &\sim(-6427d_1+19882d_2-2511d_3+24452d_4-11942d_5-7047d_6 \\
 &\qquad\qquad+7415d_7+d_8)v_1+(d_1+d_4+d_7)v_2.
\end{split}
\end{equation*}
The divisor is principal if and only $52497$ and $3$ divide the
coefficients of $v_1$ and $v_2$, respectively.

\begin{Theorem} \label{theorem: power of 2} Let $k\ge 3$ and $N=2^k$.
  Let $a$ be a generator of the cyclic group $(\Z/2^k\Z)^\times/\pm1$.
  For $\ell\ge 2$, set $\phi_\ell=\phi(2^\ell)/2=2^{\ell-2}$. Then a
  basis for $\FF_1^\infty(2^k)$ modulo $\C^\times$ is given by
$$
  \begin{cases}
  f_i=\frac{E_{a^{i-1}}E_{a^{i+\phi_{k-1}}}}
    {E_{a^i}E_{a^{i+\phi_{k-1}-1}}}, &i=1,\ldots,\phi_k-\phi_{k-1}-1, \\
  f_i=\frac{E_{a^{i-1}}^2}{E_{a^{i+\phi_{k-1}-1}}^2},&i=\phi_k-\phi_{k-1}, \\
  f_i=\frac{E_{a^{i-1}}^{(2^{k-1})}(2\tau)}
    {E_{a^{i+\phi_{k-2}-1}}^{(2^{k-1})}(2\tau)},
    &i=\phi_k-\phi_{k-1}+1,\ldots,\phi_k-\phi_{k-2}, \\
  \qquad\vdots & \qquad\vdots \\
  f_i=\frac{E_{a^{i-1}}^{(8)}(2^{k-3}\tau)}
    {E_{a^i}^{(8)}(2^{k-3}\tau)}, &i=\phi_k-1.
  \end{cases}
$$
\end{Theorem}

\begin{proof} Let $M$ be the matrix whose $(i,j)$-entry is
  $2^{k-1}B_2(a^{i+j-2}/2^k)$. The proof follows exactly the same way
  as the odd prime power case, except for that the matrices $V_\ell$
  and $U_\ell$ in \eqref{equation: Vell} and \eqref{equation: Uell}
  are defined only for $3\le\ell\le k$. Then the first $\phi_k-1$ rows of
  $U_3\ldots U_{k-1}U_kM$ will be the orders of the functions $f_i$,
  $i=1,\ldots,\phi_k-1$, at the cusps $a^{j-1}/2^k$, while the entries
  in the last row are all $2^{k-4}B_{2,\chi_0}$. We then use the
  determinant argument to show that $\{f_i\}$ is a basis. We omit the
  details here.
\end{proof}

\noindent{\bf Example.} Let $N=32$ and $a=3$. We set
$$
  M=(M_{ij})_{i,j=1,\ldots,8}, \quad
  M_{i,j}=32B_2\left(\frac{3^{i+j-2}}{32}\right),
$$
$$
  U_5=\left(\begin{smallmatrix}
  1&0&0&0&-1&0&0&0\\0&1&0&0&0&-1&0&0\\0&0&1&0&0&0&-1&0\\0&0&0&1&0&0&0&-1\\
  2&0&0&0&2&0&0&0\\0&2&0&0&0&2&0&0\\0&0&2&0&0&0&2&0\\0&0&0&2&0&0&0&2
  \end{smallmatrix}\right), \quad
  U_5'=\left(\begin{smallmatrix}
  1&-1&0&0&0&0&0&0\\0&1&-1&0&0&0&0&0\\0&0&1&-1&0&0&0&0\\0&0&0&2&0&0&0&0\\
  0&0&0&0&1&0&0&0\\0&0&0&0&0&1&0&0\\0&0&0&0&0&0&1&0\\0&0&0&0&0&0&0&1
  \end{smallmatrix}\right),
$$
$$
  U_4=\left(\begin{smallmatrix}
  1&0&0&0&0&0&0&0\\0&1&0&0&0&0&0&0\\0&0&1&0&0&0&0&0\\0&0&0&1&0&0&0&0\\
  0&0&0&0&1&0&-1&0\\0&0&0&0&0&1&0&-1\\0&0&0&0&2&0&2&0\\0&0&0&0&0&2&0&2
  \end{smallmatrix}\right), \quad
  U_3=\left(\begin{smallmatrix}
  1&0&0&0&0&0&0&0\\0&1&0&0&0&0&0&0\\0&0&1&0&0&0&0&0\\0&0&0&1&0&0&0&0\\
  0&0&0&0&1&0&0&0\\0&0&0&0&0&1&0&0\\0&0&0&0&0&0&1&-1\\0&0&0&0&0&0&2&2
  \end{smallmatrix}\right).
$$
Then we have
$$
  Q=U_3U_4U_5'U_5M=\begin{pmatrix}
  1&3&-2&5&-1&-3&2&-5\\3&-2&5&-1&-3&2&-5&1\\
  -2&5&-1&-3&2&-5&1&3\\3&-7&-5&1&-3&7&5&-1\\
  3&1&-3&-1&3&1&-3&-1\\1&-3&-1&3&1&-3&-1&3\\
  2&-2&2&-2&2&-2&2&-2\\c&c&c&c&c&c&c&c\end{pmatrix}, \quad c=-1/3,
$$
where the first $7$ rows correspond to the order of the functions
$$
  \frac{E_1E_{13}}{E_3E_{15}},\ \frac{E_3E_7}{E_9E_{13}}, \
  \frac{E_9E_{11}}{E_5E_7},\ \frac{E_5^2}{E_{11}^2},\
  \frac{E_1^{(16)}(2\tau)}{E_7^{(16)}(2\tau)},\
  \frac{E_3^{(16)}(2\tau)}{E_5^{(16)}(2\tau)},\
  \frac{E_1^{(8)}(4\tau)}{E_3^{(8)}(4\tau)}
$$
at cusps $3^{j-1}/32$, $j=1,\ldots,8$. From this we deduce that the
class number is $2^6\cdot3^2\cdot5\cdot97=279360$, as expected.

Let $P_j$, $j=1,\ldots,8$, denote the cusps $3^{j-1}/32$.
We now determine the structure of the divisor class group.
Essentially, this amounts to computing the Hermite normal form for
$Q$. Let $Q'$ be the $7\times 8$ matrix formed by deleting the last
row of $Q$. Then there is a unimodular matrix $U$ such that
$$
  UQ'=\begin{pmatrix}
  1&0&0&0&1&2&-4754&4750\\
  0&1&0&0&0&3&-5336&5332\\
  0&0&1&0&0&-4&-3865&3868\\
  0&0&0&1&0&-4&2536&-2533\\
  0&0&0&0&2&2&-354&350\\
  0&0&0&0&0&12&552&-564\\
  0&0&0&0&0&0&11640&-11640\end{pmatrix}.
$$
From this we see that the divisor class group
$\DD_1^\infty(32)/\div\FF_1^\infty(32)$ is isomorphic to
$C_{11640}\times C_{12}\times C_2$, where each component is generated by
$$
  v_1=(P_7)-(P_8), \ v_2=(P_6)+46(P_7)-47(P_8), \
  v_3=(P_5)+(P_6)-177(P_7)+175(P_8),
$$
respectively. Moreover, for a divisor $\sum_{i=1}^8d_i(P_i)$ of degree
$0$, we have
\begin{equation*}
\begin{split}
  \sum_{i=1}^8 d_i(P_i)
 &\sim(4623d_1+5474d_2+3681d_3-2720d_4+223d_5-46d_6+d_7)v_1 \\
 &\qquad +(-d_1-3d_2+4d_3+4d_4-d_5+d_6)v_2+(-d_1+d_5)v_3,
\end{split}
\end{equation*}
and it is principal if and only if the three coefficients are
congruent to $0$ modulo $11640$, $12$, and $2$, respectively.
\end{subsection}
\end{section}

\begin{section}{Non-prime power cases} \label{section: non-prime power}
\begin{subsection}{Squarefree cases} \label{subsection: squarefree}
Here we consider squarefree composite cases.

\begin{Theorem} \label{theorem: squarefree} Let $N$ be a composite
  squarefree integer, and set
  \begin{equation} \label{equation: S}
    S=\{g_1,\ldots,g_{\phi(N)/2}\}=\{g:~1\le g\le N/2,~(g,N)=1\}.
  \end{equation}
  For each integer $g$ in $S$ and each proper divisor $k$
  of $N$, define $g(k)$ to be the unique integer satisfying
  $$
    \begin{cases}
    g(k)\equiv 0&\mod k, \\
    g(k)\equiv\pm g&\mod N/k,
    \end{cases}
  $$
  in the range $1\le g(k)\le N/2$. For $g\in S$, set
  \begin{equation} \label{equation: Fg}
    F_g^{(N)}(\tau)=F_g(\tau)=\prod_{k|N,k\neq N}
    E_{g(k)}^{(N)}(\tau)^{\mu(k)},
  \end{equation}
  where $\mu(k)$ is the Mobius function. Then a basis for
  $\FF_1^\infty(N)$ modulo scalars is
  $F_{g_i}^{(N)}/F_{g_{i+1}}^{(N)}$, $i=1,\ldots,\phi(N)/2-1$.
\end{Theorem}

\begin{proof}
%
  We first show that any quotient $F_{g_i}/F_{g_j}$ of two functions
  $F_g$ satisfies the orbit condition \eqref{condition: DO}, which by
  Theorem B implies that $F_{g_i}/F_{g_j}$ is in $\FF_1^\infty(N)$.

  Let $p$ be a prime divisor of $N$ and $g$ be an integer relatively
  prime to $N$. Since $N$ is squarefree, we may write $F_g$ as
  $$
    F_g=\prod_{p\nmid k}E_{g(k)}^{\mu(k)}
         \prod_{p|k,k\neq N}E_{g(k)}^{\mu(k)}
    =E_{g(N/p)}^{\mu(N/p)}\prod_{p\nmid k,k\neq N/p}
      E_{g(k)}^{\mu(k)}E_{g(pk)}^{-\mu(k)}.
  $$
  For divisors $k$ of $N$ that are not divisible by $p$, the integers
  $g(k)$ and $g(pk)$ satisfy
  $$
    \begin{cases} g(k)\equiv 0 &\mod k, \\
      g(k)\equiv\pm g&\mod N/k, \end{cases} \qquad
    \begin{cases} g(pk)\equiv 0&\mod pk, \\
      g(pk)\equiv\pm g &\mod N/(pk).\end{cases}
  $$
  Combining these congruences, we find
  $$
    g(k)\equiv\pm g(pk)\Mod N/p.
  $$
  Therefore, $F_{g_i}/F_{g_j}$ satisfies \eqref{condition: DO}.
  We now show that $\FF_1^\infty(N)$ is generated by
  $F_{g_i}/F_{g_{i+1}}$, $i=1,\ldots,\phi(N)/2-1$. It suffices to
  prove that every function in $\FF_1^\infty(N)$ is a product of
  $F_g$.

  By Theorem B, every function in $\FF_1^\infty(N)$ is of the form
  $f(\tau)=\prod_g E_g^{e_g}$ up to a scalar, where $e_g$ satisfy
  $$
    \sum_{g\in\O_{a,p}}e_g=0
  $$
  for each prime divisor $p$ of $N$ and each integer $a$. Now let $p$
  be a prime divisor of $N$ and $g$ be an integer in the range $1\le
  g\le N/2$ satisfying $(g,N)=p$. Consider the set
  $$
   T=\O_{g,p}=\{1\le h\le N/2:~h\equiv\pm g\mod N/p\}.
  $$
  Except for $g$ itself, all elements of $T$ are relatively prime to
  $N$. It follows that for $g$ with $(g,N)=p$,
  \begin{equation} \label{equation: theorem 4 1}
    e_{g}=-\sum_{h\equiv\pm g\Mod N/p,~(h,N)=1}e_h,
  \end{equation}
  where $h$ runs over all integers in the range $1\le h\le N/2$
  satisfying the stated conditions.

  Likewise, if $p_1$ and $p_2$ are two distinct prime divisors of $N$,
  then for all $g$ with $1\le g\le N/2$ and $(g,N)=p_1p_2$, the set
  $$
    T=\{1\le h\le N/2:~h\equiv\pm g\mod N/(p_1p_2)\}
  $$
  can be partitioned into a union of $4$ disjoint subsets
  $$
    T=T_1\cup T_{p_1}\cup T_{p_2}\cup T_{p_1p_2},
  $$
  where
  $$
    T_k=\{h\in T:~(h,N)=k\}.
  $$
  Then condition \eqref{condition: DO} yields
  $$
    \sum_{h\in T_1}e_h+\sum_{h\in T_{p_1}}e_h+\sum_{h\in T_{p_2}}e_h
   +\sum_{h\in T_{p_1p_2}}e_h=\sum_{h\in T}e_h=0.
  $$
  Now the set $T_{p_1p_2}$ consists of $g$ itself.
  Furthermore, by \eqref{equation: theorem 4 1}, we have for all $h\in
  T_{p_i}$, $i=1,2$,
  $$
    e_h=-\sum_{\ell\equiv\pm h\Mod N/p_i,~(\ell,N)=1}e_\ell.
  $$
  It follows that
  $$
    \sum_{h\in T_{p_i}}e_h=-\sum_{\ell\in T_1}e_\ell
  $$
  since for each $\ell\in T_1$ there exists exactly one element $h\in
  T_{p_i}$ such that $h\equiv\pm\ell\mod N/p_i$. Summarizing, we find,
  for all $g$ with $(g,N)=p_1p_2$.
  $$
    e_g=\sum_{h\in T_1}e_h=
    \sum_{h\equiv\pm g\Mod N/p_1p_2,~(h,N)=1}e_h.
  $$

  In general, following the same argument, we can prove by induction
  that if $p_1,\ldots,p_k$ are distinct prime factors of $N$, then for
  all $g$ with $1\le g\le N/2$, we have
  \begin{equation} \label{equation: theorem 4 dependence}
    e_g=(-1)^{\mu(\gcd(g,N))}\sum_{h\equiv\pm g\Mod N/(g,N),~(h,N)=1}e_h,
  \end{equation}
  where the summation runs over all integers $h$ satisfying $1\le
  h\le N/2$ and the stated conditions. From this we see that
  \begin{equation} \label{equation: theorem 4 dependence 2}
  \begin{split}
  f(\tau)&=\prod_g E_g^{e_g}
      =\prod_{d|N,~d\neq N}\prod_{(g,N)=d}E_g^{e_g}
   =\prod_{d|N,~d\neq N}\prod_{(g,N)=d}
    \prod_{\substack{h\equiv\pm g\Mod N/d \\(h,N)=1}}E_g^{\mu(d)e_h} \\
  &=\prod_{(h,N)=1}\prod_{d|N,~d\neq N}E_{h(d)}^{\mu(d)e_h}
   =\prod_{(h,N)=1}F_h^{e_h}.
  \end{split}
  \end{equation}
  This completes the proof of the theorem.
\end{proof}

\noindent{\bf Example.} Consider $N=42$. Following
  Theorem \ref{theorem: squarefree}, we set
  \begin{equation*}
  \begin{split}
    F_1&=\frac{E_1E_6E_{14}E_{21}}{E_{20}E_{15}E_7}, \quad
    F_5=\frac{E_5E_{12}E_{14}E_{21}}{E_{16}E_9E_7}, \quad
    F_{11}=\frac{E_{11}E_{18}E_{14}E_{21}}{E_{10}E_3E_7}, \\
    F_{13}&=\frac{E_{13}E_6E_{14}E_{21}}{E_8E_{15}E_7}, \quad
    F_{17}=\frac{E_{17}E_{18}E_{14}E_{21}}{E_4E_3E_7}, \quad
    F_{19}=\frac{E_{19}E_{12}E_{14}E_{21}}{E_2E_9E_7}.
  \end{split}
  \end{equation*}
  According to Theorem \ref{theorem: squarefree}, the group
  $\FF_1^\infty(42)$, up to scalars, is generated by
  $$
    f_1=\frac{F_1}{F_5}, \ f_2=\frac{F_5}{F_{11}}, \
    f_3=\frac{F_{11}}{F_{13}}, \ f_4=\frac{F_{13}}{F_{17}}, \
    f_5=\frac{F_{17}}{F_{19}}.
  $$
  To check the correctness, we form a $6\times 6$ matrix
  $$
    M=\begin{pmatrix}5&9&-5&6&-14&-1\\6&-8&-1&2&12&-11\\
    5&-6&9&-14&-1&5\\ 8&-12&-2&11&-6&1\\
    -2&11&-12&-6&1&8\\1&1&1&1&1&1\end{pmatrix}
  $$
  whose first $5$ rows are the orders of $f_i$ at $1/42$, $5/42$,
  $11/42$, $13/42$, $17/42$, and $19/42$, respectively. We have
  $\det(M)/6=248430$, which agrees with the class number obtained from
  Yu's formula.

  To determine the group structure of generators of the divisor class
  group, we remove the last row of $M$ and put it in the Hermite
  normal form. Explicitly, we have
  $$
    \begin{pmatrix}-4&-53&37&71&22\\0&6&-2&-6&-1\\
    -3&-33&25&46&15\\6&46&-42&-71&-26\\
    11&142&-99&-190&-59\end{pmatrix}M
   =\begin{pmatrix}1&0&0&1&-1021&1019\\
    0&1&0&-20&109&-90\\0&0&1&-18&-640&657\\
    0&0&0&91&910&-1001\\0&0&0&0&2730&-2730\end{pmatrix},
  $$
  where the $5\times 5$ matrix on the left is a unimodular matrix.
  (We have removed the last row of $M$.) From this we see that the
  divisor class group is isomorphic to $C_{2730}\times C_{91}$, whose
  components are generated by the classes of
  $$
    (17/42)-(19/42), \qquad(13/42)+10(17/42)-11(19/42),
  $$
  respectively.
\medskip

\noindent{\bf Remark.} We remark that one can actually use Theorem
  \ref{theorem: squarefree} to give another proof of Theorem A for
  squarefree integers $N$. The key is the distribution relation
  \eqref{equation: Bernoulli relation} of Lemma \ref{lemma: Bernoulli
  relation}. For example, let us consider the case $N=21$.

  Let $b_{g,a}$ denote the numbers $21B_2(ag/21)/2$, which is the
  order of $E_g$ at $a/21$ when $(a,21)=1$. By Lemma
  \ref{lemma: Bernoulli relation}, we have
  \begin{equation*}
  \begin{split}
    3(b_{1,a}+b_{8,a}+b_{6,a})&=b_{3,a}, \\
    3(b_{2,a}+b_{9,a}+b_{5,a})&=b_{6,a}, \\
    3(b_{4,a}+b_{10,a}+b_{3,a})&=b_{9,a},
  \end{split}
  \end{equation*}
  and
  $$
    7(b_{1,a}+b_{4,a}+b_{7,a}+b_{10,a}+b_{8,a}+b_{5,a}+b_{2,a})=b_{7,a}
  $$
  for all integers $a$. From these relations, we obtain
  $$
    \begin{pmatrix}b_{6,a}\\ b_{9,a}\\b_{3,a}\end{pmatrix}
  =-\frac1{3^3-1}\begin{pmatrix}27&3&9\\9&27&3\\3&9&27\end{pmatrix}
    \begin{pmatrix} b_{1,a}+b_{8,a} \\ b_{2,a}+b_{5,a} \\
    b_{4,a}+b_{10,a}\end{pmatrix},
  $$
  and
  $$
    b_{7,a}=-\frac7{7-1}(b_{1,a}+b_{2,a}+b_{4,a}+b_{5,a}
   +b_{8,a}+b_{10,a}).
  $$
  Set
  $$
    F_1=\frac{E_1}{E_6E_7},~F_2=\frac{E_2}{E_9E_7},
   ~F_4=\frac{E_4}{E_3E_7},~F_5=\frac{E_5}{E_9E_7},
   ~F_8=\frac{E_8}{E_6E_7},~F_{10}=\frac{E_{10}}{E_3E_7}.
  $$
  Now let $M$ be the $6\times 6$ matrix whose $(i,j)$-entry is the
  order $21B_2(2^{i+j-2}/21)/2$ of $E_{2^{i-1}}$ at the cusp
  $2^{j-1}/21$. Then the orders of $F_1,F_2,F_4,F_8,F_5,F_{10}$ at the
  cusps $2^{j-1}/21$ will be
  $$
    (V_1-V_3-V_7)M,
  $$
  where $V_1$ is the identity matrix of size $6$,
  $$
    V_3=-\frac1{3^3-1}\begin{pmatrix} W_3&W_3\\ W_3&W_3\end{pmatrix},
    \qquad
    W_3=\begin{pmatrix}27&3&9\\9&27&3\\3&9&27\end{pmatrix},
  $$
  and $V_7$ is the $6\times 6$ matrix whose entries are all $-7/6$.
  Then the orders of the generators $F_1/F_2$, $F_2/F_4$, $F_4/F_8$,
  $F_8/F_5$, $F_5/F_{10}$ will make up the first $5$ rows of
  $$
   Q=\begin{pmatrix}1&-1&0&0&0&0\\0&1&-1&0&0&0\\0&0&1&-1&0&0\\
    0&0&0&1&-1&0\\0&0&0&0&1&-1\\0&0&0&0&0&1\end{pmatrix}
   (V_1-V_3)(V_1-V_7)M.
  $$
  (Note that $V_3V_7M$ has the same number $7/2$ in every entry. Thus,
  the term $V_3V_7M$ does not contribute anything to the first $5$
  rows of $Q$.) By a direct computation and Lemma \ref{lemma:
  Bernoulli matrix}, the determinant of $Q$ is equal to
  $$
    \det Q=\cdot(1+3^3)\cdot(1+7)\cdot\prod_{\chi}\frac14B_{2,\chi},
  $$
  where $\chi$ runs over all even Dirichlet characters modulo $21$.
  We then check that the sum of the entries in the last row of $Q$ is
  $$
    16=\frac12(1+3)(1+7)=\frac14B_{2,\chi_0}(1+3)(1+7).
  $$
  Thus, by Lemma \ref{lemma: lattice index}, the class number
  $h_1^\infty(21)$ is equal to
  $$
    \frac1{16}\det Q=\frac{1+3^3}{1+3}\cdot\frac{1+7}{1+7}\cdot
    \prod_{\chi\neq\chi_0}\frac14B_{2,\chi},
  $$
  which is \eqref{equation: Yu formula 2} for $N=21$.

  In general, if $N=p_1\ldots p_n$ is a squarefree integer with $n\ge
  2$, we may deduce Yu's formula using the same argument as above.
  Let $a_1,\ldots,a_{\phi(N)/2}$ be the integers in the range $1\le
  a_i\le N/2$ that are relatively prime to $N$. Let $M$ be the matrix
  whose $(i,j)$-entry is $NB_2(a_ia_j/N)/2$ so that the $i$th row of
  $M$ encodes the order of $E_{a_i}$ at $a_j/N$. Now for each divisor
  $k$ of $N$, using the distribution relation \eqref{equation:
  Bernoulli relation}, we can record the order of $E_{a_i(k)}$ at
  $a_j/N$ in a matrix of the form $V_kM$. Then the order of $F_{a_i}$
  at $a_j/N$ will be the $(i,j)$-entry of
  $$
    \left(\sum_{k|N,~k\neq N}\mu(k)V_k\right)M
  $$
  We then show that the matrices $V_k$ satisfy
  $$
    V_{k_1}V_{k_2}=V_{k_2}V_{k_1}
  $$
  for divisors $k_1$ and $k_2$ of $N$ that are relatively prime and
  $$
    \det(V_1-V_{p_i})=(1+f_{p_i})^{e_{p_i}},
  $$
  where $f_{p_i}$ and $e_{p_i}$ are defined as in Theorem A'. Then
  following the argument in the special case $N=21$ above, one can
  deduce Theorem A' for composite squarefree integers $N$.
\end{subsection}

\begin{subsection}{Remaining cases} \label{subsection: nonsquarefree}
In this section we give a basis $\GG$ for $\FF_1^\infty(N)$ for
non-squarefree integers $N$ that are not prime powers.

Let $L=\prod_{p|N}p$ be the product of distinct prime divisors of $N$.
For each divisor $M$ of $N$ that is a multiple of $L$, we will
construct a set $\GG^{(M)}$ of modular functions in $\FF_1^\infty(M)$
so that the union
$$
  \bigcup_{M:~M|N,~L|M}\{f(N\tau/M):~f(\tau)\in\GG^{(M)}\}
$$
forms a basis for $\FF_1^\infty(N)$ modulo scalars. The definition of
$\GG^{(M)}$ is given as follows.

When $M=L$ is squarefree, we define the set $\GG^{(L)}$ to be the
basis for $\FF_1^\infty(L)$ given in Theorem \ref{theorem:
squarefree}. When $M\neq L$, $M$ is non-squarefree. Let
$p_1,\ldots,p_\ell$ be the prime divisors of $M$ such that $p_i^2|M$,
$i=1,\ldots,\ell$, and let $p_{\ell+1},\ldots,p_k$ be the other prime
factors of $M$. Set
$$
  K=\prod_{i=\ell+1}^k p_i.
$$
($K=1$ if $M$ is squarefull.) For an integer $g$ relatively prime to
$M$ and a divisor $k$ of $K$, we let $g(k)$ be the unique integer
satisfying
$$
  \begin{cases}g(k)\equiv 0,&\mod k, \\
  g(k)\equiv\pm g, &\mod M/k, \end{cases}
$$
in the range $1\le g(k)\le M/2$, and set
\begin{equation} \label{equation: theorem 5 F}
  F_g^{(M)}(\tau)=\prod_{k|K}E_{g(k)}^{(M)}(\tau)^{\mu(k)}.
\end{equation}

For each integer $g$ in the range $1\le g<M/(2\prod_{i=1}^\ell p_i)$
satisfying $(g,M)=1$ and each integer $m_i$ with $1\le m_i\le p_i-1$,
$i=1,\ldots,\ell$, we set
\begin{equation} \label{equation: theorem 5 G}
  G_{g,m_1,\ldots,m_\ell}^{(M)}(\tau)=\prod_{(n_1,\ldots,n_\ell)\in
  \{0,1\}^\ell}F_{g+n_1m_1M/p_1+\cdots+n_\ell m_\ell M/p_\ell}^{(M)}
  (\tau)^{(-1)^{n_1+\cdots+n_\ell}}.
\end{equation}
Then we define the set $\GG^{(M)}$ to be the set of all such functions
$$
  \GG^{(M)}=\left\{G^{(M)}_{g,m_1,\ldots,m_\ell}(\tau):\,
  1\le g\le\frac{M}{2\prod_{i=1}^\ell p_i},\,(g,L)=1,\,1\le m_i\le
  p_i-1 \right\}.
$$
\medskip

\noindent{\bf Example.} Let $M=180=2^2\cdot3^2\cdot5$. We have
  $p_1=2$, $p_2=3$, $p_3=5$, and $K=5$. Then $\GG^{(M)}$ consists of
  $8$ functions $G_{1,1,m}^{(M)}$, $G_{7,1,m}^{(M)}$,
  $G_{11,1,m}^{(M)}$, and $G_{13,1,m}^{(M)}$ with $m=1,2$. Among them,
  the function $G_{1,1,2}^{(M)}$ is defined as
  $$
    G_{1,1,2}^{(M)}=\frac{F_1^{(M)}F_{1+M/2+2M/3}^{(M)}}
    {F^{(M)}_{1+M/2}F^{(M)}_{1+2M/3}}
   =\frac{F_1^{(M)}F_{31}^{(M)}}{F_{91}^{(M)}F_{121}^{(M)}},
  $$
  where $F_1^{(M)}=E_1/E_{35}$, $F^{(M)}_{31}=E_{31}/E_5$,
  $F^{(M)}_{91}=E_{91}/E_{55}$, and
  $F_{121}^{(M)}=E_{59}/E_{85}$. Thus, we have
  $$
    G_{1,1,2}^{(M)}=\frac{E_1E_{31}E_{55}E_{85}}{E_{91}E_{59}E_{35}E_5}.
  $$
\medskip

The definition of $\GG^{(M)}$ stems from the following observations.

\begin{Lemma} \label{lemma: origin 1} Let the notations $K$, $L$,
  $M$, $p_1,\ldots,p_\ell$, and $p_{\ell+1},\ldots,p_k$ be given as
  above. For any $g$ relatively prime to $M$ and any $j$ with $1\le
  j\le\ell$, the elements in the orbit
  $$
    \O_{g,p_1\ldots p_j}=\{h\mod M:~g\equiv h\mod M/(p_1\ldots p_j)\}
    /\pm 1
  $$
  can be uniquely represented by
  \begin{equation} \label{equation: origin 1}
    g+m_1\frac M{p_1}+\cdots+m_j\frac M{p_j}, \quad 0\le m_i\le p_i-1.
  \end{equation}
\end{Lemma}

\begin{proof} Write $P=p_1\ldots p_j$. We first show that
  $|\O_{g,P}|=P$, that is, all numbers $g+kM/P$ are distinct 
  under the identification $\Z/M\Z\to(\Z/M\Z)/\pm 1$. Suppose that
  $g\equiv-g-kM/P\mod M$ for some $k$. Then we have $2g\equiv
  0\mod M/P$, which is impossible since $M/P\ge p_1\ldots p_\ell\ge
  6$ and $g$ is assumed to be relatively prime to $M$.

  Now it is obvious that every number in \eqref{equation:
  origin 1} is congruent to $g$ modulo $M/P$, and
  the number of elements in \eqref{equation: origin 1} is the same as
  $|\O_{g,P}|=P$. Thus, we only need to show that two different
  elements in \eqref{equation: origin 1} can not be congruent to each
  other modulo $M/P$. This can be achieved by considering the
  reduction modulo $p_i^{n_i-1}$ for various $p_i$, where $p_i^{n_i}$
  denotes the exact power of $p$ dividing $M$.
\end{proof}

\begin{Lemma} \label{lemma: origin} Let all the notations be given as
  above. Assume that $\prod_g E_g^{e_g}\in\FF_1^\infty(M)$. Then, for
  any integer $g$ relatively prime to $M$ and any integer $j$ with
  $1\le j\le \ell$, we have the relation
  \begin{equation} \label{equation: lemma origin}
    e_g =\sum_{P|p_1\ldots p_j}\mu(P)\sum_{h\in\O_{g,P}}e_h
    =(-1)^j\sum_{m_1=1}^{p_1-1}\cdots\sum_{m_j=1}^{p_j-1}
    e_{g+m_1M/p_1+\cdots+m_jM/p_j}
  \end{equation}
  for $e_g$.
\end{Lemma}

\begin{proof} We start by considering the first equality in
  \eqref{equation: lemma origin}. For the case $j=1$, it is an
  immediate consequence of the orbit condition \eqref{condition: DO}.
  We then proceed by induction.

  Assume that the first equality of \eqref{equation: lemma origin}
  holds up to $j$ with $j<\ell$, that is,
  $$
    e_g=\sum_{P|p_1\ldots p_j}\mu(P)\sum_{h\in\O_{g,P}}e_h.
  $$
  Then we have, by the induction hypothesis,
  $$
    e_g=\sum_{P|p_1\ldots p_j}\mu(P)\sum_{h\in\O_{g,P}}
    \left(\sum_{k\in\O_{h,1}}e_k-\sum_{k\in\O_{h,p_{j+1}}}e_k\right).
  $$
  The orbit $\O_{h,1}$ contains $h$ itself. Also, by Lemma \ref{lemma:
  origin 1} above, we have
  $$
    \sum_{h\in\O_{g,P}}\sum_{k\in\O_{h,p_{j+1}}}e_k
   =\sum_{h\in\O_{g,p_{j+1}}P}e_h.
  $$
  It follows that
  $$
    e_g=\sum_{P|p_1\ldots p_j}\mu(P)\left(
    \sum_{h\in\O_{g,P}}e_h-\sum_{h\in\O_{g,p_{j+1}P}}e_h\right)
   =\sum_{P|p_1\ldots p_{j+1}}\mu(P)\sum_{h\in\O_{g,P}}e_h.
  $$
  This proves the first equality of \eqref{equation: lemma origin}.
  We now prove the second equality.

  Regardless of what $P$ is, the elements $h$ in $\O_{g,P}$ always take
  the form
  $$
    h=g+m_1\frac M{p_1}+\cdots+m_j\frac M{p_j}
  $$
  for some $m_i$ with $0\le m_i<p_i-1$. In the other direction, such
  an element $h$ can appear in $\O_{g,P}$ if and only if for all $i$
  with $m_i\neq 0$, $p_i$ divides $P$. Thus, let $Q=\prod_{i:m_i\neq
  0}p_i$. Then the coefficient of $e_h$ in \eqref{equation: lemma
  origin} is equal to
  $$
    \sum_{P|p_1\ldots p_j,Q|P}\mu(P),
  $$
  which is $(-1)^j$ if $Q=p_1\ldots p_j$ and $0$ if $Q\neq p_1\ldots
  p_j$. This gives the second equality of \eqref{equation: lemma
  origin}.
\end{proof}


We now verify that $G_{g,m_1,\ldots,m_\ell}^{(M)}(\tau)$ are modular
functions contained in $\FF_1^\infty(M)$.

\begin{Lemma} \label{lemma: G} Let the notations be given as
  above. Then the functions $G_{g,m_1,\ldots,m_\ell}^{(M)}$ are
  modular functions contained in $\FF_1^\infty(M)$.
\end{Lemma}

\begin{proof} The case when $M=L$ is squarefree is verified in Theorem
  \ref{theorem: squarefree}. For other $M$, using the argument given
  in the first paragraph of the proof of Theorem
  \ref{theorem: squarefree}, we see that for all prime factors $q$ of
  $M$ satisfying $q^2\nmid M$, the functions $F_g$ all satisfy
  condition \eqref{condition: DO}. For prime factors $p$ of $M$
  satisfying $p^2|M$, it is clear from our definition that
  $G_{g,m_1,\ldots,m_\ell}^{(M)}$ satisfy condition
  \eqref{condition: DO}. Therefore, by Theorem B, the functions
  $G_{g,m_1,\ldots,m_\ell}^{(M)}$ are all modular functions contained
  in $\FF_1^\infty(M)$.
\end{proof}

\begin{Corollary} \label{corollary: nonsquarefree} Let all the
  notations be given as above. Define
  \begin{equation} \label{equation: GG}
    \GG^{(M)}(d)=\{f(d\tau):~f(\tau)\in\GG^{(M)}\}.
  \end{equation}
  Then all the functions in $\GG^{(M)}(N/M)$ are modular functions
  belonging to $\FF_1^\infty(N)$.
\end{Corollary}

\begin{proof} This immediately follows from Lemmas \ref{lemma: level
  lowering 2} and \ref{lemma: G}.
\end{proof}

Now we present our basis for $\FF_1^\infty(N)$.

\begin{Theorem} \label{theorem: nonsquarefree} Let all the notations
  be given as above. Then a basis for $\FF_1^\infty(N)$ modulo
  $\C^\times$ is
  $$
    \bigcup_{M:~M|N,~L|M}\GG^{(M)}(N/M),
  $$
  where $\GG^{(M)}(N/M)=\{f(N\tau/M):~f(\tau)\in\GG^{(M)}\}$.
\end{Theorem}

\noindent{\bf Example.} Consider the case $N=p^aq^b$, where $p$ and
$q$ are distinct primes and $a,b\ge 1$. Let us compute the number of
elements in $\GG^{(p^iq^j)}$ for $i,j\ge 1$.

When $2\le i\le a$ and $2\le j\le b$, the set
\begin{equation*}
\begin{split}
  \GG^{(p^iq^j)}=\{G_{g,m_1,m_2}^{(p^iq^j)}:
  &~1\le g\le p^{i-1}q^{j-1}/2,~(g,pq)=1, \\
  &\qquad1\le m_1\le p-1,~1\le m_2\le q-1\}
\end{split}
\end{equation*}
has
$$
  \frac12\phi(p^{i-1}q^{j-1})\cdot(p-1)\cdot(q-1)
 =\frac12p^{i-2}q^{j-2}(p-1)^2(q-1)^2
$$
elements. When $i=1$ and $2\le j\le b$, the set
$$
  \GG^{(pq^j)}=\{G_{g,m}^{(pq^j)}:~1\le g\le pq^{j-1}/2,
  ~(g,pq)=1,~1\le m\le q-1\}
$$
has $q^{j-2}(p-1)(q-1)^2/2$ elements. Likewise, when $2\le i\le a$ and
$j=1$, the set $\GG_{p^iq}$ has $p^{i-2}(p-1)^2(q-1)$ elements.
When $i=1$ and $j=1$, the set $\GG^{(pq)}$ has $(p-1)(q-1)/2-1$ elements.
Thus, the set
$$
  \bigcup_{M:~M|p^aq^b,~pq|M}\GG^{(M)}(p^aq^b/M)
$$
has totally
\begin{equation*}
\begin{split}
 &\frac12\sum_{i=2}^a\sum_{j=2}^bp^{i-2}q^{j-2}(p-1)^2(q-1)^2
 +\frac12\sum_{j=2}^b q^{j-2}(p-1)(q-1)^2 \\
 &\qquad\qquad+\frac12\sum_{i=2}^a
   p^{j-2}(p-1)^2(q-1)+\frac{(p-1)(q-1)}2-1 \\
 &\qquad=\frac12(p^{a-1}-1)(q^{b-1}-1)(p-1)(q-1)
 +\frac12(q^{b-1}-1)(p-1)(q-1) \\
 &\qquad\qquad+\frac12(p^{a-1}-1)(p-1)(q-1)+\frac12(p-1)(q-1)-1 \\
 &\qquad=\frac12p^{a-1}q^{b-1}(p-1)(q-1)-1,
\end{split}
\end{equation*}
which is the precisely the number of functions needed to generate
$\FF_1^\infty(p^aq^b)$.

\begin{proof}[Proof of Theorem \ref{theorem: nonsquarefree}] By
  Corollary \ref{corollary: nonsquarefree}, the functions in
  $\GG^{(M)}(N/M)$ are all contained in $\FF_1^\infty(N)$. To show
  that they generate $\FF_1^\infty(N)$, we recall Theorem B that
  $f(\tau)\in\FF_1^\infty(N)$ if and only if
  $$
    f(\tau)=\prod_g E_g(\tau)^{e_g}
  $$
  is a product of $E_g(\tau)$ with the exponents $e_g$
  satisfying the orbit condition \eqref{condition: DO}. We will prove
  that such a function can be expressed as a product of functions from
  $\cup_M\GG^{(M)}(N/M)$.

  Let $L=\prod_{p|N}p$ be the product of distinct prime divisors of
  $N$. We start out by observing that, for each divisor $d$ of $N/L$,
  the set
  $$
    S_d=\{g\ \Mod N:~\gcd(g,N/L)=d\}/\pm 1
  $$
  is stable under the map $g\mod N\mapsto g+N/p\mod N$ for all prime
  divisor $p$ of $N$ since $N/p$ is always a multiple of $N/L$. Thus,
  by Theorem B, if $f(\tau)=\prod E_g^{e_g}$ is a modular function in
  $\FF_1^\infty(N)$, then so is
  $$
    \prod_{g\in S_d}E_g^{e_g}
  $$
  for each divisor $d$ of $N/L$. Therefore, to prove the theorem, it
  suffices to consider the special case where
  $f(\tau)\in\FF_1^\infty(N)$ takes the form $\prod_{g\in
  S_d}E_g^{e_g}$. We claim that such a function can
  be expressed as a product of functions from $\GG^{(N/d)}(d)$.
  
  Assume that $\prod_{g\in S_d}E_g^{e_g}\in\FF_1^\infty(N)$. Then
  $e_g$ satisfy condition \eqref{condition: DO}
  $$
    \sum_{g\in\O_{a,p}}e_g=0
  $$
  for all $a$ with $\gcd(a,N/L)=d$ and all prime factors $p$
  of $N$. This condition can also be written as
  $$
    \sum_{g/d\equiv\pm b\mod (N/d)/p}e_g=0,
  $$
  for all $b$ satisfying $\gcd(b,N/(dL))=1$. Therefore,
  from Theorem B we deduce that
  $\prod_{g\in S_d}E_g^{(N)}(\tau)^{e_g}\in\FF_1^\infty(N)$ if and only if
  $\prod_{g\in S_d}E_{g/d}^{(N/d)}(\tau)^{e_g}\in\FF_1^\infty(N/d)$.
  Thus, the assertion that every function of the form $\prod_{g\in
  S_d}E_g^{(N)}(\tau)^{e_g}$ in $\FF_1^\infty(N)$ is generated by
  $\GG^{(N/d)}(d)$ is equivalent to the assertion that every function
  of the form
  $$
    \prod_{h:~(h,N/(dL))=1}E_h^{(N/d)}(\tau)^{e_h}
  $$
  in $\FF_1^\infty(N/d)$ is generated by functions from
  $\GG^{(N/d)}(1)=\GG^{(N/d)}$.

  Rehashing our problem, what we need to show now is the following.
  Let $M$ be a non-squarefree, non-prime power integer. Let
  $L=\prod_{p|M}p$ be the product of distinct prime divisors of $M$.
  Let $p_1,\ldots,p_\ell$ be the prime divisors of $M$ such that
  $p_i^2|M$ and $p_{\ell+1},\ldots,p_k$ be the remaining prime factors
  of $M$. Set
  $$
    K=\prod_{i=\ell+1}^k p_i.
  $$
  We are required to show that if a function $f(\tau)$ in
  $\FF_1^\infty(M)$ takes the form
  $$
    f(\tau)=\prod_{g:~(g,M/L)=1}E_g^{(M)}(\tau)^{e_g},
  $$
  then it is a product of functions from
  $$
    \GG^{(M)}=\left\{G^{(M)}_{g,m_1,\ldots,m_\ell}(\tau):\,
    1\le g\le\frac{M}{2\prod_{i=1}^\ell p_i},\,(g,L)=1,\,1\le m_i\le
    p_i-1 \right\},
  $$
  where $G^{(M)}_{g,m_1,\ldots,m_\ell}(\tau)$ are defined by
  \eqref{equation: theorem 5 F} and \eqref{equation: theorem 5 G}.

  First of all, following the deduction of \eqref{equation: theorem 4
  dependence} in the proof of Theorem \ref{theorem: squarefree}, we
  find that the exponents $e_g$ satisfy
  $$
    e_g=(-1)^{\mu(\gcd(g,K))}\sum_{h\equiv\pm g\Mod
    M/(g,K),~(h,M)=1}e_h,
  $$
  where $h$ runs over all integers in the range $1\le h\le M/2$
  satisfying the stated congruence condition. Then the argument in
  \eqref{equation: theorem 4 dependence 2} gives
  \begin{equation} \label{equation: theorem 5 proof 1}
    f(\tau)=\prod_{(g,M)=1}F_g^{(M)}(\tau)^{e_g},
  \end{equation}
  where $F_g^{(M)}$ is defined by \eqref{equation: theorem 5 F}.

  For convenience, we drop the superscript $(M)$ and write
  $F_g^{(M)}(\tau)$ as $F_g$. Partitioning the product in
  \eqref{equation: theorem 5 proof 1} according to the orbits
  $\O_{g,p_1\ldots p_\ell}$, we have
  $$
    f(\tau)=\prod_{\substack{1\le g\le M/(2p_1\ldots p_\ell) \\
    (g,M)=1}}\prod_{h\in\O_{g,p_1\ldots p_\ell}}
    F_h^{e_h}.
  $$
  By Lemma \ref{lemma: origin 1}, every element $h$ in
  $\O_{g,p_1\ldots p_\ell}$ can be uniquely represented as
  \begin{equation} \label{equation: theorem 5 proof 3}
    h=g+m_1\frac M{p_1}+\cdots+m_\ell\frac M{p_\ell}.
  \end{equation}
  For such an element, we define
  $$
    Q(h)=\prod_{i:~m_i=0}p_i.
  $$
  By Lemma \ref{lemma: origin}, we have
  $$
    e_h=\sum_{P|Q(h)}\mu(P)\sum_{k\in\O_{h,P}}e_k.
  $$
  Notice that the second equality in \eqref{equation: lemma origin}
  shows that the coefficient of $e_k$ on the right-hand side of the
  above expression is nonzero if and only if the numbers $n_i$ in
  \begin{equation} \label{equation: theorem 5 proof 2}
    k=g+n_1\frac M{p_1}+\cdots+n_\ell\frac M{p_\ell}
  \end{equation}
  are all nonzero. Thus, we may write $f(\tau)$ as
  $$
    f(\tau)=\prod_g\prod_k\left(\prod_{P|p_1\ldots p_\ell}
    \left(\prod_{h:P|Q(h),h\in\O_{k,P}}F_h\right)^{\mu(P)}\right)^{e_k},
  $$
  where $g$ runs over all integers satisfying $1\le g\le M/(2p_1\ldots
  p_\ell)$ and $(g,M)=1$ and $k$ runs over all numbers of the form
  \eqref{equation: theorem 5 proof 2} with $n_i\neq 0$ for all $i$.
  Now consider the product over $h$. An integer $h$ of the form
  \eqref{equation: theorem 5 proof 3} satisfies $P|Q(h)$ if and only
  if $p_i|P$ implies $m_i=0$. Also, $h\in\O_{k,P}$ if and only if
  $m_j=n_j$ for all $j$ with $p_j\nmid P$. Therefore, the only $h$ that
  satisfies both $P|Q(h)$ and $h\in\O_{k,P}$ is
  $$
    h=g+\sum_{1\le i\le\ell,p_i\nmid P}n_i\frac M{p_i}.
  $$
  Then we have
  \begin{equation*}
  \begin{split}
   &\prod_{P|p_1\ldots p_\ell}\left(\prod_{h:P|Q(h),h\in\O_{k,P}}F_h
    \right)^{\mu(P)} \\
  &\qquad\quad=\prod_{(r_1,\ldots,r_\ell)\in\{0,1\}^\ell}
    F_{g+r_1n_1M/p_1+\cdots+r_\ell n_\ell M/p_\ell}^{(M)}
    (\tau)^{(-1)^{n_1+\cdots+n_\ell}}
   =G^{(M)}_{g,n_1,\ldots,n_\ell}(\tau),
  \end{split}
  \end{equation*}
  and
  $$
    f(\tau)=\prod_g\prod_{1\le n_i\le p_i-1}
    G^{(M)}_{g,n_1,\ldots,n_\ell}
    (\tau)^{e_{g+n_1M/p_1+\cdots+n_\ell M/p_\ell}}.
  $$
  This completes the proof of the theorem.
\end{proof}

\noindent{\bf Example.}
Let $N=36$. In the notations of Theorem \ref{theorem:
  nonsquarefree}, we have $\GG^{(6)}=\emptyset$ since $\phi(6)/2-1=0$.
Also, when $M=12$, we have $\ell=1$, $p_1=2$, and
\begin{equation*}
\begin{split}
  \GG^{(12)}(3)=\{G^{(12)}_{g,m}(3\tau):~1\le g\le 12/4,~(g,12)=1,~1\le
  m\le 1\}=\{G^{(12)}_{1,1}(3\tau)\},
\end{split}
\end{equation*}
where
\begin{equation} \label{equation: example 36-1}
  G^{(12)}_{1,1}(3\tau)=\frac{F^{(12)}_1(3\tau)}{F^{(12)}_7(3\tau)}
 =\frac{E^{(12)}_1(3\tau)/E^{(12)}_3(3\tau)}
  {E^{(12)}_5(3\tau)/E^{(12)}_3(3\tau)}
 =\frac{E^{(12)}_1(3\tau)}{E^{(12)}_5(3\tau)}.
\end{equation}
When $M=18$, we have $\ell=1$, $p_1=3$, and
\begin{equation*}
\begin{split}
  \GG^{(18)}(2)&=\{G^{(18)}_{g,m}(2\tau):~1\le g\le 18/6,~(g,18)=1,~1\le
  m\le 2\} \\
  &=\{G^{(18)}_{1,1}(2\tau),G^{(18)}_{1,2}(2\tau)\},
\end{split}
\end{equation*}
where
\begin{equation} \label{equation: example 36-2}
  G^{(18)}_{1,1}(2\tau)=\frac{F^{(18)}_1(2\tau)}{F^{(18)}_7(2\tau)}
 =\frac{E^{(18)}_1(2\tau)/E^{(18)}_8(2\tau)}
  {E^{(18)}_7(2\tau)/E^{(18)}_2(2\tau)},
\end{equation}
and
\begin{equation} \label{equation: example 36-3}
  G^{(18)}_{1,2}(2\tau)=\frac{F^{(18)}_1(2\tau)}{F^{(18)}_{13}(2\tau)}
 =\frac{E^{(18)}_1(2\tau)/E^{(18)}_8(2\tau)}
  {E^{(18)}_5(2\tau)/E^{(18)}_4(2\tau)}.
\end{equation}
When $M=36$, we have $\ell=2$, $p_1=2$, $p_2=3$, and
\begin{equation*}
\begin{split}
  \GG^{(36)}&=\{G^{(36)}_{g,m_1,m_2}(\tau):~1\le g\le 36/12,~(g,36)=1,
  ~1\le m_1\le 1,~1\le m_2\le 2\} \\
  &=\{G^{(36)}_{1,1,1}(\tau),G^{(36)}_{1,1,2}(\tau)\},
\end{split}
\end{equation*}
where
\begin{equation} \label{equation: example 36-4}
  G^{(36)}_{1,1,1}(\tau)=\frac{F^{(36)}_1(\tau)F^{(36)}_{31}(\tau)}
  {F^{(36)}_{19}(\tau)F^{(36)}_{13}(\tau)}
 =\frac{E^{(36)}_1(\tau)E^{(36)}_5(\tau)}
  {E^{(36)}_{17}(\tau)E^{(36)}_{13}(\tau)},
\end{equation}
and
\begin{equation} \label{equation: example 36-5}
  G^{(36)}_{1,1,2}(\tau)=\frac{F^{(36)}_1(\tau)F^{(36)}_{43}(\tau)}
  {F^{(36)}_{19}(\tau)F^{(36)}_{25}(\tau)}
 =\frac{E^{(36)}_1(\tau)E^{(36)}_7(\tau)}
  {E^{(36)}_{17}(\tau)E^{(36)}_{11}(\tau)}.
\end{equation}
By Theorem \ref{theorem: nonsquarefree}, the functions
\eqref{equation: example 36-1}--\eqref{equation: example 36-5} form a
basis for $\FF_1^\infty(36)$. To check the correctness, we form a
$5\times 6$ matrix
$$
  M=\begin{pmatrix}3&-3&-3&3&3&-3\\
  6&-4&-2&-2&-4&6\\4&2&-6&-6&2&4\\
  6&1&5&-5&-1&-6\\5&6&-1&1&-6&-5\end{pmatrix}.
$$
whose rows consist of the orders of the above
functions at $1/36$, $5/36$, $7/36$, $11/36$, $13/36$, and $17/36$.
From the matrix we deduce that the class number
$h_1^\infty(36)$ is equal to $31248$, which agrees with what one gets
using Theorem A. To determine the group structure and the generators of
the divisor class group, we compute the Hermite normal form of $M$. We
find it is
$$
  \begin{pmatrix}1&0&0&0&-1540&1539 \\
  0&1&0&0&-981&980\\0&0&1&-1&2044&-2044\\
  0&0&0&4&-1588&1584\\0&0&0&0&7812&-7812\end{pmatrix}.
$$
Therefore the divisor class group is isomorphic to $C_4\times
C_{7812}$, where the components are generated by the classes of
$$
  (11/36)-397(13/36)+396(17/36), \qquad (13/36)-(17/36),
$$
respectively.
\medskip

\noindent{\bf Example.}
Consider the case $N=40$ with $L=10$. We have
$$
  \GG^{(10)}(4)=\left\{\frac{E^{(10)}_1(4\tau)E^{(10)}_2(4\tau)}
  {E^{(10)}_3(4\tau)E^{(10)}_4(4\tau)}\right\}, \qquad
  \GG^{(20)}(2)=\left\{\frac{E^{(20)}_1(2\tau)}{E^{(20)}_9(2\tau)},
 ~\frac{E^{(20)}_3(2\tau)}{E^{(20)}_7(2\tau)}\right\},
$$
and
$$
  \GG^{(40)}(1)=\left\{\frac{E_1(\tau)E_5(\tau)}
  {E_{15}(\tau)E_{19}(\tau)},
  \frac{E_3(\tau)E_{15}(\tau)}
  {E_{17}(\tau)E_5(\tau)},
  \frac{E_7(\tau)E_5(\tau)}
  {E_{13}(\tau)E_{15}(\tau)},
  \frac{E_9(\tau)E_5(\tau)}
  {E_{11}(\tau)E_{15}(\tau)}
  \right\}.
$$
We find that the class number $h_1^\infty(40)$ is $4775680$, and the
divisor class group is isomorphic to $C_{298480}\times C_4^2$.
\medskip

\noindent{\bf Example.}
Consider $N=72$. We have $\GG^{(6)}=\emptyset$,
$\GG^{(12)}(6)=\{E^{(12)}_1(6\tau)/E^{(12)}_5(6\tau)\}$,
$$
  \GG^{(18)}(4)=\left\{\frac{E^{(18)}_1(4\tau)E^{(18)}_2(4\tau)}
  {E^{(18)}_7(4\tau)E^{(18)}_8(4\tau)},
 ~\frac{E^{(18)}_1(4\tau)E^{(18)}_4(4\tau)}
  {E^{(18)}_5(4\tau)E^{(18)}_8(4\tau)}\right\},
$$
$$
  \GG^{(24)}(3)=\left\{\frac{E^{(24)}_1(3\tau)E^{(24)}_3(3\tau)}
  {E^{(24)}_{11}(3\tau)E^{(24)}_9(3\tau)},
 ~\frac{E^{(24)}_5(3\tau)E^{(24)}_9(3\tau)}
  {E^{(24)}_7(3\tau)E^{(24)}_3(3\tau)}\right\},
$$
$$
  \GG^{(36)}(2)=\left\{\frac{E^{(36)}_1(2\tau)E^{(36)}_5(2\tau)}
  {E^{(36)}_{13}(2\tau)E^{(36)}_{17}(2\tau)},
 ~\frac{E^{(36)}_1(2\tau)E^{(36)}_7(2\tau)}
  {E^{(36)}_{11}(2\tau)E^{(36)}_{17}(2\tau)}\right\},
$$
and
$$
  \GG^{(72)}=\left\{\frac{E_1(\tau)E_{11}(\tau)}
  {E_{25}(\tau)E_{35}(\tau)},
  \frac{E_1(\tau)E_{13}(\tau)}
  {E_{23}(\tau)E_{25}(\tau)},
  \frac{E_5(\tau)E_7(\tau)}
  {E_{29}(\tau)E_{31}(\tau)},
  \frac{E_5(\tau)E_{17}(\tau)}
  {E_{19}(\tau)E_{31}(\tau)}\right\}.
$$
Using this basis for $\FF_1^\infty(72)$, we find the divisor class
group is isomorphic to
$$
  C_4\times C_{12}\times C_{36}\times C_{144}\times C_{9146133360}.
$$
\end{subsection}
\end{section}

\begin{section}{Computational results} \label{section: results}
In this section, we give a few tables of computational results.
The first table contains the group structure of $\CC_1^\infty(N)$ for
$N\le 100$. (Note that for $N=1,\ldots,10$ and $N=12$, the Jacobian is
trivial.) For the reader's convenience, we have also included the
genus of the modular curve $X_1(N)$ and the prime factorization of the
group order of $\CC_1^\infty(N)$ for $N\le 50$. We have used Hazama's
formula (whenever applicable) and Yu's formula to check that
the group orders are correct. Here the notation
$[n_1,\ldots,n_k]$ means that the group structure is
$(\Z/n_1\Z)\times\cdots\times(\Z/n_k\Z)$.

In the second table, we give the $p$-parts of $\CC_1^\infty(p^n)$ for
small $p^n$. The notation $(p^{e_1})^{n_1}\ldots(p^{e_k})^{n_k}$ means
that the primary decomposition of $\CC_1^\infty(p^n)$ contains $n_i$
copies of $\Z/p^{e_i}\Z$. Finally, in the last table, we list the
$p$-parts of $\CC_1^\infty(mp^n)$ for selected integers $N=mp^n$.
\vfill

\begin{table}
\caption{Group structure of $\CC_1^\infty(N)$, $N\le 100$}
$$ \extrarowheight1pt
\begin{array}{c||c|l|l} \hline\hline
N & \text{genus} & \text{class number} & \text{structure} \\ \hline\hline
11 & 1 & 5 & \text{cyclic} \\ \hline
13 & 2 & 19 & \text{cyclic} \\ \hline
14 & 1 & 3 & \text{cyclic} \\ \hline
15 & 1 & 2^2 & \text{cyclic} \\ \hline
16 & 2 & 2\cdot5 & \text{cyclic} \\ \hline
17 & 5 & 2^3\cdot 73 & \text{cyclic} \\ \hline
18 & 2 & 7 & \text{cyclic} \\ \hline
19 & 7 & 3^2\cdot487 & \text{cyclic} \\ \hline
20 & 3 & 2^2\cdot5 & \text{cyclic} \\ \hline
21 & 5 & 2\cdot7\cdot13 & \text{cyclic} \\ \hline
22 & 6 & 5\cdot31 & \text{cyclic} \\ \hline
23 & 12 & 11\cdot37181 & \text{cyclic} \\ \hline
24 & 5 & 2^2\cdot3\cdot5 & \text{cyclic} \\ \hline
25 & 12 & 5\cdot71\cdot641 & \text{cyclic} \\ \hline
26 & 10 & 3\cdot5\cdot7\cdot19 & \text{cyclic} \\ \hline
27 & 13 & 3^3\cdot19\cdot307 & [3, 52497] \\ \hline
28 & 10 & 2^6\cdot3\cdot13 & [4, 4, 156] \\ \hline
29 & 22 & 2^6\cdot3\cdot7\cdot43\cdot17837 & [4, 4, 64427244]
  \\ \hline
30 & 9  & 2^2\cdot5\cdot17 & \text{cyclic} \\ \hline
31 & 26 & 2^2\cdot5^2\cdot7\cdot11\cdot2302381 &
  [10, 1772833370] \\ \hline
32 & 17 & 2^6\cdot3^2\cdot5\cdot97 & [2, 12, 11640] \\ \hline
33 & 21 & 2\cdot3\cdot5\cdot11\cdot61\cdot421 & \text{cyclic} \\ \hline
34 & 21 & 2^3\cdot3\cdot5^2\cdot17\cdot73 & [5, 148920] \\
  \hline
35 & 25 & 2^2\cdot3\cdot5\cdot13^2\cdot31\cdot37\cdot61
   & [13, 54574260] \\ \hline
36 & 17 & 2^4\cdot3^2\cdot7\cdot31 & [4, 7812] \\ \hline
37 & 40 & 3^2\cdot5\cdot7\cdot19\cdot37\cdot73\cdot577\cdot17209
  & \text{cyclic} \\ \hline
38 & 28 & 3^4\cdot7\cdot73\cdot487 & [9, 2239713]\\ \hline
39 & 33 & 2^5\cdot3^4\cdot7^2\cdot13^3\cdot19 & [1638, 3236688] \\ \hline
40 & 25 & 2^8\cdot5\cdot7\cdot13\cdot41
   & [4, 4, 298480] \\ \hline
41 & 51 & 2^4\cdot5\cdot13\cdot31^2\cdot431\cdot250183721
   & \text{cyclic} \\ \hline
42 & 25 & 2\cdot3\cdot5\cdot7^2\cdot13^2 & [91, 2730] \\ \hline
43 & 57 & 2^2\cdot7\cdot19\cdot29\cdot463\cdot1051\cdot416532733
   & [2, 1563552532984879906] \\ \hline
44 & 36 & 2^8\cdot5\cdot7\cdot31^2\cdot101 & [4, 4, 124, 438340] \\ \hline
45 & 41 & 2^2\cdot3^6\cdot7\cdot31\cdot73\cdot3637
   & [9, 9, 2074093812] \\ \hline
46 & 45 & 11\cdot89\cdot683\cdot37181 & \text{cyclic} \\ \hline
47 & 70 & 23\cdot139\cdot82397087\cdot12451196833
  & \text{cyclic} \\ \hline
48 & 37 & 2^8\cdot3\cdot5^2\cdot41\cdot73 & [4, 20, 718320] \\ \hline
49 & 69 & 7^3\cdot113\cdot2437\cdot1940454849859
   & [7, 26183855453442042671] \\ \hline
50 & 48 & 5^2\cdot11\cdot31\cdot41\cdot71\cdot641 & \text{cyclic} \\ \hline
   \hline
\end{array}
$$
\vfill\newpage
\end{table}

\begin{table*}
$$ \extrarowheight3pt
\begin{array}{c||l} \hline\hline
N & \text{structure} \\ \hline\hline
51 & [8, 1201887101691040] \\ \hline
52 & [4, 4, 4, 4, 5823652380] \\ \hline
53 & [182427302879183759829891277] \\ \hline
54 & [9, 509693373] \\ \hline
55 & [110, 8972396739917886000] \\ \hline
56 & [4, 4, 16, 16, 528, 4427280] \\ \hline
57 & [7, 3446644128227394822] \\ \hline
58 & [4, 172, 4622976893220] \\ \hline
59 & [17090415233025974812945896997681] \\ \hline
60 & [4, 4, 80, 2174640] \\ \hline
61 & [77, 11245132002040993823541663395815] \\ \hline
62 & [11, 10230, 1813608537510] \\ \hline
63 & [9, 9, 18, 23940, 36513252544860] \\ \hline
64 & [2, 4, 4, 4, 8, 48, 73910454036960] \\ \hline
65 & [2, 4, 4, 4, 4, 64, 69171648, 3833806702270272] \\ \hline
66 & [341, 8669648790] \\ \hline
67 & [661, 228166524544404715482454653548693117] \\ \hline
68 & [2, 4, 4, 4, 4, 4, 340, 29034225327840] \\ \hline
69 & [419621485489110883825078452] \\ \hline
70 & [13, 39, 7825676012700] \\ \hline
71 & [701, 846772703911192558471548563811885556615] \\ \hline
72 & [4, 12, 36, 144, 9146133360] \\ \hline
73 & [2, 2, 9940318318931388769396722069876040037329842] \\ \hline
74 & [87381, 1137260725252531425] \\ \hline
75 & [25, 25, 229987489818652358805100] \\ \hline
76 & [4, 4, 4, 4, 4, 4, 36, 939850887824333604] \\ \hline
77 & [4, 8, 152, 456, 89190700421406700205983720927320] \\ \hline
78 & [273, 1638, 66303553680] \\ \hline
79 & [521, 29427100164209457485089447933533181481550301759] \\ \hline
80 & [4, 4, 4, 4, 16, 80, 13855590960585920] \\ \hline
81 & [3, 9, 9, 9, 9, 9, 27, 87945822520529641810558635771] \\ \hline
82 & [155, 8525, 729055927995792711600] \\ \hline
83 & [98686349372029170201616572533501298687049100544333193] \\ \hline
84 & [4, 4, 4, 364, 3640, 2254324800] \\ \hline
85 & [8, 16, 16, 6089864235465097347758333448185021417280] \\
\end{array}
$$
\end{table*}
\vfill

\begin{table*}
$$ \extrarowheight3pt
\begin{array}{c||l} \hline\hline
N & \text{structure} \\ \hline\hline
86 & [127, 32766, 25615681147891287499998] \\ \hline
87 & [4, 8, 8, 950886641657525419895591717028046920] \\ \hline
88 & [4, 4, 4, 4, 4, 16, 496, 5456, 271296378433899280] \\ \hline
89 & [2, 2,
  16752498873229395124991547026173083931082925441759332846930] \\
  \hline
90 & [3, 9, 63, 552341552604660] \\ \hline
91 & [4, 8, 24, 191520, 679547136977329249210045615114328726306400]
  \\ \hline
92 & [4, 4, 4, 4, 4, 4, 4, 4, 4, 4, 17196616295174096057972420]
  \\ \hline
93 & [2, 58520, 27655854146903254636965396728546654040] \\ \hline
94 & [1636915575475523620555957536005041] \\ \hline
95 & [234, 234, 28721474819561261359171101174478296073945013520]
  \\ \hline
96 & [4, 4, 4, 4, 48, 240, 297091975708501440] \\ \hline
97 & [35,
  184448350307327781233828921084300852328775499446274957379700139920]
   \\\hline
98 & [49, 783386490641061468174212253579] \\ \hline
99 & [3, 9, 9, 9, 9, 9, 9, 9, 117459, 
            186692207903601450412662961290630] \\ \hline
100 & [4, 4, 4, 4, 4, 20, 100, 100, 11049959065582110305500] \\
  \hline\hline
\end{array}
$$
\end{table*}

\begin{table}
\caption{$p$-primary part of $\CC_1^\infty(p^n)$}
$$ \extrarowheight3pt
\begin{array}{c||l} \hline\hline
p^n & p\text{-primary subgroups} \\ \hline\hline
2^4 & (2) \\ \hline
2^5 & (2)(2^2)(2^3) \\ \hline
2^6 & (2)(2^2)^3(2^3)(2^4)(2^5) \\ \hline
2^7 & (2)(2^2)^7(2^3)(2^4)^3(2^5)(2^6)(2^7)\\ \hline
2^8 & (2)(2^2)^{15}(2^3)(2^4)^7(2^5)(2^6)^3(2^7)(2^8)(2^9) \\ \hline
2^9 & (2)(2^2)^{31}(2^3)(2^4)^{15}(2^5)(2^6)^7(2^7)(2^8)^3(2^9)(2^{10})(2^{11})
  \\ \hline\hline
3^3 & (3)(3^2) \\ \hline
3^4 & (3)(3^2)^5(3^3)(3^4)\\ \hline
3^5 & (3)(3^2)^{17}(3^3)(3^4)^5(3^5)(3^6) \\ \hline
3^6 & (3)(3^2)^{53}(3^3)(3^4)^{17}(3^5)(3^6)^5(3^7)(3^8)\\ \hline\hline
5^2 & (5) \\ \hline
5^3 & (5)(5^2)^7(5^3)\\ \hline
5^4 & (5)(5^2)^{39}(5^3)(5^4)^7(5^5)\\ \hline\hline
7^2 & (7)(7^2) \\ \hline
7^3 & (7)(7^2)^{17}(7^3)(7^4) \\ \hline\hline
\end{array}
$$
\end{table}

\begin{table}
\caption{$p$-primary part of $\CC_1^\infty(mp^n)$}
$$ \extrarowheight3pt
\begin{array}{c||l} \hline\hline
mp^n & p\text{-primary subgroups} \\ \hline\hline
2\cdot3^2 & (1) \\ \hline
2\cdot3^3 & (3^2)^2 \\ \hline
2\cdot3^4 & (3^2)^6(3^4)^2 \\ \hline
2\cdot3^5 & (3^2)^{18}(3^4)^6(3^6)^2 \\ \hline
2\cdot3^6 & (3^2)^{54}(3^4)^{18}(3^6)^6(3^8)^2 \\ \hline
2\cdot5^2 & (5^2) \\ \hline
2\cdot5^3 & (5^2)^8(5^4) \\ \hline
2\cdot5^4 & (5^2)^{40}(5^4)^8(5^6) \\ \hline
2\cdot7 & (1) \\ \hline
2\cdot7^2 & (7^2)^2 \\ \hline
2\cdot7^3 & (7^2)^{18}(7^4)^2 \\ \hline\hline
3\cdot2^3 & (2^2) \\ \hline
3\cdot2^4 & (2^2)^2(2^4) \\ \hline
3\cdot2^5 & (2^2)^4(2^4)^2(2^6) \\ \hline
3\cdot2^6 & (2^2)^8(2^4)^4(2^6)^2(2^8) \\ \hline
3\cdot2^7 & (2^2)^{16}(2^4)^8(2^6)^4(2^8)^2(2^{10}) \\ \hline
3\cdot2^8 & (2^2)^{32}(2^4)^{16}(2^6)^8(2^8)^4(2^{10})^2(2^{12}) \\
\hline
3\cdot2^9 & (2^2)^{64}(2^4)^{32}(2^6)^{16}(2^8)^8(2^{10})^4
  (2^{12})^2(2^{14}) \\ \hline\hline
4\cdot3^2 & (3^2) \\ \hline
4\cdot3^3 & (3^2)^4(3^4) \\ \hline
4\cdot3^4 & (3^2)^{12}(3^4)^4(3^6) \\ \hline
4\cdot3^5 & (3^2)^{36}(3^4)^{12}(3^6)^6(3^8) \\ \hline\hline
6\cdot5 & (5) \\ \hline
6\cdot5^2 & (5)(5^2)^2(5^3) \\ \hline
6\cdot5^3 & (5)(5^2)^{15}(5^3)(5^4)^2(5^5) \\ \hline
6\cdot7 & (7)^2 \\ \hline
6\cdot7^2 & (7)^2(7^2)^3(7^3)^2 \\ \hline
6\cdot7^3 & (7)^2(7^2)^{34}(7^3)^2(7^4)^3(7^5)^2 \\ \hline
\hline
\end{array}
$$
\end{table}
\end{section}

\bibliographystyle{plain}

\end{document}